\documentclass[
	english
]{scrartcl}

\usepackage[T1]{fontenc}
\usepackage[utf8]{inputenc}
\usepackage{a4wide}
\usepackage{amsmath,amsfonts,amsthm,amssymb}
\usepackage[colorlinks,citecolor=blue]{hyperref}
\usepackage{doi,natbib}
\usepackage{natbib}
\usepackage{cancel}
\usepackage{relsize}
\usepackage{bbm}

\usepackage{tikz}
\usetikzlibrary{trees,arrows}
\usetikzlibrary{automata}

\tikzstyle{bnode}=[circle ,draw=black,fill=black,thick, inner sep=0pt ,minimum size=2mm]
\tikzstyle{fnode}=[circle ,draw=black,fill=black,thick, inner sep=0pt ,minimum size=4mm]

\usepackage{changes}

\usepackage{enumitem}
\setlist[enumerate]{label=(\alph*)}

\usepackage[figure]{hypcap}
\usepackage{graphicx}
\graphicspath{{./img/}}
\usepackage{subcaption} 

\usepackage{preamble}

\usepackage{marginnote}

\usepackage[capitalise,nameinlink]{cleveref}

\crefname{figure}{Figure}{Figures}
\crefname{table}{Table}{Tables}

\allowdisplaybreaks

\numberwithin{equation}{section}

\newcommand\norm[1]{\left\Vert#1\right\Vert}
\newcommand\nnorm[1]{\Vert#1\Vert}

\newcommand\innerprod[2]{\left( #1, #2\right)}

\newcommand\N{\mathbb{N}}
\newcommand\R{\mathbb{R}}

\newcommand\dx{\mathrm{d}\omega}

\newcommand{\diag}{\operatorname{diag}}

\newcommand{\dt}{\mathrm d t}
\newcommand{\ds}{\mathrm d s}
\renewcommand{\dx}{\mathrm d x}

\DeclareMathAlphabet{\mathpzc}{OT1}{pzc}{m}{it}

\newtheorem{theorem}{Theorem}[section]

\newtheorem{proposition}[theorem]{Proposition}

\newtheorem{remark}[theorem]{Remark}
\newtheorem{definition}[theorem]{Definition}

\definecolor{mygreen}{rgb}{0.0,0.7,0.0}
\definecolor{mybrown}{rgb}{0.5,0.5,0.0}

\newcommand\changed[1]{\textcolor{black}{#1}}

\begin{document}

\title{%
		Inverse demand tracking in transportation networks
	}%
\author{%
	Simone G\"{o}ttlich%
	\footnote{%
		University of Mannheim,
		School of Business Informatics and Mathematics,
		68159 Mannheim,
		Germany,
		\email{goettlich@uni-mannheim.de},
		\orcid{0000-0002-8512-4525}
	}%
	\and
	Patrick Mehlitz%
	\footnote{%
		Philipps-Universit\"at Marburg,
		Department of Mathematics and Computer Science,
		35032 Marburg,
		Germany,
		\email{mehlitz@uni-marburg.de},
		\orcid{0000-0002-9355-850X}
		}%
	\and
	Thomas Schillinger%
	\footnote{%
		University of Mannheim,
		School of Business Informatics and Mathematics,
		68159 Mannheim,
		Germany,
		\email{schillinger@uni-mannheim.de},
		\orcid{0000-0002-9524-327X}
	}%
	}
\date{\today}
\publishers{}
\maketitle

\begin{abstract}
	This paper deals with the reconstruction of the desired demand
	in an optimal control problem, stated over a tree-shaped transportation network which is
	governed by a linear hyperbolic conservation law.
	As desired demands typically undergo fluctuations 
	due to seasonality or unexpected events making 
	short-term adjustments necessary, such an approach can exemplary be used
	for forecasting from past data.
	We suggest to model this problem as a so-called inverse optimal control problem, i.e.,
	a hierarchical optimization problem whose inner problem is the optimal
	control problem and whose outer problem is the reconstruction problem. 
	In order to guarantee the existence of solutions in the function space framework,
	the hyperbolic conservation law is interpreted in weak sense allowing for
	control functions in Lebesgue spaces.
	\changed{%
		For the computational treatment of the model,
		we transfer the hierarchical problem into a nonsmooth single-level one
		by plugging the uniquely determined solution of the inner optimal control problem
		into the outer reconstruction problem
		before applying techniques from nonsmooth optimization.
	}%
	Some numerical experiments are presented to visualize various features of the
	model including different types of noise in the demand 
	and strategies of how to observe the network
	in order to obtain good reconstructions of the desired demand.
\end{abstract}

\begin{keywords}	
	Inverse optimal control, Linear hyperbolic conservation laws, Transportation networks
\end{keywords}

\begin{msc}	
	\mscLink{49J20}, \mscLink{65M32}, \mscLink{90C33}, \mscLink{90C35}
\end{msc}

\section{Introduction}\label{sec:introduction}

Flow problems over energy and supply networks model a broad range of interesting applications,
see \cite{BressanCanicGaravelloHertyPiccoli2014} for a survey.
In this paper, we investigate transportation networks of tree shape where the flow on edges
is modeled, for simplicity, via (linear) hyperbolic conservation laws, 
as typically used for electric transmission lines \cite{GoettlichHertySchillen2016}, 
heating networks \cite{ReinMohringDammKlar2020},
or networks of gas pipelines \cite{BandaHertyKlar2006,GugatSchultzWintergerst2018}. 
A control function is used to model
the inflow at some source vertex, and the aim of optimization is to choose this function 
in such a way that certain desired demands at the sinks of the network are tracked as close
as possible.  
As mentioned in some recent contributions, see
\cite{GoettlichKornLux2019,GoettlichSchillinger2022a,GoettlichSchillinger2022b},
these desirable demands are subject to perturbations, noise, or other sources of stochasticity. 
In the aforementioned papers, this issue has been faced by modeling the problem as a
stochastic optimal control problem which is influenced by randomness via appropriately
chosen stochastic processes.

In this paper, we are concerned with related phenomena.
Let us consider the following practically relevant situation. 
There exists a company (C2) which appoints a second company (C1) to deliver a certain amount of 
electricity/heat/gas at the demand vertices over time by inserting the requested product at
the source of the network over time. In this regard, C1 has to solve the optimal control problem mentioned
above. We now enrich the considered situation by assuming that there is a network operator (NO),
different from C1 and C2, which partially observes the flow along the network and, depending
on this, charges C1 and C2 to pay some tax for employing the network.
As outlined above, the desired demands requested by C2 are subject to stochastic influences and, additionally,
may vary due to a seasonal behavior.
From past data, NO now wants to forecast the desired demand of C2 and the associated actions of C1, 
exemplary for fixing taxes to plan future income. Typically, NO is not aware of the 
desired demand as he only observes the actual network flow along some but, most likely, not all
edges of the network (as it might be expensive to equip the overall network with sensors or to run
them on each edge over all time). Furthermore, the forecasting model should be capable of
recognizing seasonal behavior of the desired demands as it is exemplarily presented for an electricity market in \cite{CoskunKorn2021}.

In order to model this situation, we consider it from the viewpoint of inverse optimal control,
i.e., we aim to identify parameters in an optimal control problem (and not only in a dynamical system). 
Here, the optimal control problem of interest is the aforementioned network flow problem,
and the appearing desired demand plays the role of this parameter.
We assume that we are given observed (but, most likely, noisy)
pairs of optimal inflow and optimal network flow, and aim to reconstruct the desired demands which are modeled as a
convex combination of given ansatz functions. It is, thus, our goal to find the associated
weight parameters which characterize a suitable standard (periodically emerging) choice for the desired demand.
As we are interested in the robustness of our approach, 
we consider additional perturbations in the model and study different types of temporal restrictions in the observation of the network
to evaluate whether these are sufficient for good forecasting.

Naturally, the model of interest is a hierarchical optimization problem with two decision levels.
Coming back to our exemplary situation from above, at the outer (or upper-level) problem, the NO
is in position to partially observe the network and chooses certain weights, which then give a
tangible desired demand. At the inner (or lower-level) problem, C1 now can solve the network
flow problem. Along those parts of the network, which are observed by NO, the latter can
compare the past data and the real-time data obtained from the inner problem for this particular
choice of the weight parameters. Noting that this decision order leads to a well-posed problem,
NO aims to choose the weight parameters in such a way that past data and real-time data
match as good as possible. 
As our model has two decision levels, it is a so-called bilevel optimization problem.

For more than 50 years, bilevel optimization is a major field of research in mathematical programming
due to numerous underlying applications e.g.\ in data science, economy, finance, machine learning, or natural sciences,
see \cite{Bard1998,Dempe2002,ShimizuIshizukaBard1997} for an introduction 
and \cite{Dempe2020} for a recent survey 
which presents an overview of contributions in this
area. Recently, bilevel optimization turned out to be of particular interest in the context of
transportation or energy networks, see e.g.\ \cite{DempeKalashnikovPerezValdesKalashnykova2015}.
This also includes the rapidly growing field of hierarchical control, see e.g.\ \cite{MehlitzWachsmuth2020} for an overview,
and, particularly, so-called inverse optimal control already mentioned earlier,
see \cite{HinzePinnauUlbrichUlbrich2009,Troeltzsch2009,Troutman1996,Vinter2010} for an introduction to the topic of optimal control.
Inverse control possesses several interesting applications e.g.\ in the context of human locomotion, see 
\cite{AlbrechtLeiboldUlbrich2012,AlbrechtUlbrich2017,AnbrechtPassenbergSobotkaPeerBussUlbrich2010,MombaurTruongLaumond2010}.
The theory on inverse optimal control including ordinary and partial differential equations addresses the existence of solutions,
optimality conditions, and solution algorithms, see e.g.\
\cite{DempeHarderMehlitzWachsmuth2019,FriedemannHarderWachsmuth2022,HarderWachsmuth2019,HatzSchloederBock2012,HollerKunischBarnard2018,SuryanSinhaMaloDeb2016} 
and is developing fast.
In abstract bilevel optimization, two decision makers, a leader and a follower, need to choose variables in order to minimize
their associated cost function which also depends on the variables of the other decision maker, respectively.
More precisely, the leader chooses his variables first which are handed over to the follower who now can solve his
optimization problem (which is parametric in the leader's variable) 
\changed{%
	to global optimality.
}%
The solutions are then given to the leader, who now can
evaluate his objective. Often, one assumes that leader and follower cooperate in order to optimize the leader's objective, and
this procedure is referred to as the optimistic approach to the problem,
see \cite{Zemkoho2016} for an overview of other approaches avoiding ill-posedness in bilevel optimization.
The leader's and follower's problem are often referred to as upper- and 
lower-level problem, respectively.
\changed{%
	As the follower has to determine globally optimal solutions of his problem by nature of bilevel optimization,
	one typically requires that the lower-level problem is convex in the follower's variable
	in order to circumvent issues related to nonconvex global optimization
	at the lower-level stage.
}%

We start our investigations by modeling the problem of interest as an inverse control
problem in \cref{sec:model_problem}. 
Therefore, we first study the existence of solutions for linear hyperbolic conservation
laws in a function space which is suitable for optimal control before setting up
the lower- and upper-level problem consecutively.
Furthermore, we demonstrate that the resulting optimization problem possesses
an optimal solution in the function space setting we are investigating.
\changed{%
	In \cref{sec:experiments}, we address the computational treatment of the model.
	\Cref{sec:numerical_treatment} describes our approach to the numerical solution
	of the problem. As it is analytically possible to compute the network
	flow associated with the input, we are in position to distill a state-reduced version of the
	parametric optimal control problem. The associated solution operator, 
	which, at least in pointwise fashion, can similarly be computed analytically due to the nice
	structure of the problem, turns out to be a nonsmooth single-valued mapping.
	Plugging the latter into the superordinate reconstruction problem 
	and performing a suitable discretization, we end up with a nonsmooth optimization 
	which we solve with the aid of MATLAB's \texttt{patternsearch} solver in default mode.
	The general set-up of our computational experiments is carved out in 
	\cref{sec:general_set_up}.
}%
Numerical results are presented in \cref{sec:documentation} in order to visualize the
effectiveness and several different features of the approach.
Particular focus is laid on the robustness of the model 
with respect to additional uncertainties, restricted observation options,
\changed{%
	and the presence of additional inflow constraints.
}%
Some concluding remarks close the paper in \cref{sec:conclusions}.

\section{The model problem}\label{sec:model_problem}

In this section, we set up the model of our interest.
First, we discuss the particular shape of the lower-level parametric optimal
control problem in \cref{sec:lower_level_problem}. Therefore, we first present
the underlying network dynamics and discuss regularity features of associated
solutions. Second, the lower-level objective function is constructed, and 
solvability of the overall lower-level problem is discussed. 
In \cref{sec:upper_level_problem}, we derive the superordinate upper-level problem
and demonstrate that it possesses an optimal solution in the function space setting. 

\subsection{The lower-level problem}\label{sec:lower_level_problem}

\changed{%
	In this subsection, we are concerned with the derivation and analysis 
	of the lower-level optimal control problem.
	To start, we state the lower-level dynamics and 
	discuss existence and uniqueness of solutions associated
	with this system. Afterwards, we set up the (parametric) lower-level problem,
	show that, for each set of parameters, it possesses a unique solution,
	and investigate properties of the associated solution operator.
}%

\subsubsection{Setting up the network and network dynamics}

We consider a directed graph $G=(V,E)$ which is a tree (in the sense that
whenever the directed edges are interpreted as undirected, then the resulting graph
would be free of cycles).
Let us use the notation $V:=\{v_0,\ldots,v_n\}$ and note that $|E|=n$ by
nature of trees.
Some more details on $G$ and the notation we are going to exploit are discussed below.
\begin{itemize}
	\item The uniquely determined source vertex of the network $G$ is $v_0\in V$.
		Furthermore, we assume that $v_0$ is a leaf of $G$, i.e., 
		there is only one edge which leaves $v_0$, and the vertex at
		its end will be denoted by $v_1$.
	\item In $V_D\subset V$, we collect all vertices which possess no outgoing edges.
		These are the demand vertices.
	\item All remaining intermediate (or inner) vertices of the network are collected
		in the set $V_I:=V\setminus (V_D\cup\{v_0\})$.
	\item For $v_i\in V\setminus\{v_0\}$, we identify the uniquely determined
		edge which ends at $v_i$ by $(i)$. 
	\item The set $E^+(i)$ is used to denote the set of all edges starting at vertex $v_i$.
		Furthermore, we use $E_D:=\{(i)\in E\,|\,v_i\in V_D\}$ 
		\changed{%
			to denote the set of edges that end at a demand vertex.
			Clearly, $|E_D|=|V_D|$.
		}%
\end{itemize}

We visualize the above notation in \cref{fig:example_network}.
\changed{%
	For the theory of this paper, 
	it is not mandatory that the vertex $v_0$ possesses just one outgoing edge. 
	One can interpret $v_0$ as an upstream supersource. 
	Besides, this additional assumption
	simplifies the notation because we can abstain from the introduction of
	distribution parameters at the inflow vertex later on.
}%

\begin{figure}[h]
 \begin{center}
	\begin{tikzpicture}[-,>=stealth',shorten >=1pt, semithick]

	\node[bnode,label=left:$v_0$] at (-2,0) (0){};
	\node[bnode,label=above:$v_1$] at (0,0) (1){};
	\node[bnode,label=above:$v_2$] at (2,1) (2){};
	\node[bnode,label=below:$v_3$] at (2,-1) (3){};
	\node[bnode,label=right:$v_4$] at (4,1) (4){};
	\node[bnode,label=right:$v_5$] at (4,0) (5){};
	\node[bnode,label=right:$v_6$] at (4,-1) (6){};
	
	\path[->] (0) edge node[above]{$(1)$} (1);
	\path[->] (1) edge node[above left]{$(2)$} (2);
	\path[->] (1) edge node[below left]{$(3)$} (3);
	\path[->] (2) edge node[above]{$(4)$} (4);
	\path[->] (2) edge node[below]{$(5)$} (5);
	\path[->] (3) edge node[above]{$(6)$} (6);
	\end{tikzpicture}
	\end{center}
	\caption{An exemplary network with $V_D=\{v_4,v_5,v_6\}$, $V_I=\{v_1,v_2,v_3\}$, $E^+(2)=\{(4),(5)\}$, and $E_D=\{(4),(5),(6)\}$.}
	\label{fig:example_network}
\end{figure}
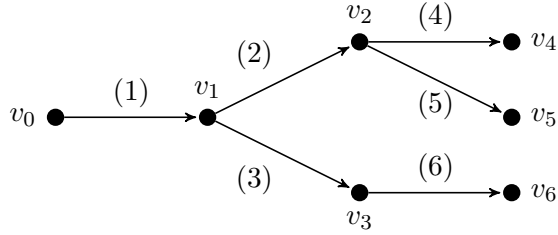

At the source $v_0$, the injection of flow over time $\mathcal T:=(0,T)$, 
where $T>0$ is the final time,
is modeled by the control variable $u\colon\mathcal T\to\R$ which has to be chosen
from an appropriate function space.

The flow over $(i)$ at time $t\in\mathcal T$ at the spatial coordinate $x\in\Omega$ 
will be denoted by $z^{(i)}(t,x)$. 
Here, we assume that $\Omega:=(0,\omega)$ is a bounded real interval.
The density has to obey the linear hyperbolic conservation law
\begin{subequations}\label{eq:network_flow}
	 \begin{align}
	 	\label{eq:network_flow_PDE}
	 		& z^{(i)}_t(t,x)+\lambda^{(i)}z^{(i)}_x(t,x)=0,\qquad
	 			\text{a.e.\ on }\mathcal T\times\Omega,\,(i)\in E,\\
	 	\label{eq:network_flow_initials}
	 		& z^{(i)}(0,x)=0,\qquad
	 			\text{a.e.\ on }\Omega,\,(i)\in E,\\
	 	\label{eq:network_flow_inflow}
	 		& \lambda^{(1)}z^{(1)}(t,0)=u(t),\qquad\text{a.e.\ on }\mathcal T,\\
	 	\label{eq:network_flow_conservation}
	 		& \lambda^{(k)}z^{(k)}(t,0)=\alpha_{i,k}\lambda^{(i)}z^{(i)}(t,\omega),\qquad
	 			\text{a.e.\ on }\mathcal T,\,v_i\in V_I,\,(k)\in E^+(i).
	 \end{align}
\end{subequations} 
Particularly, the flux functions of the conservation law are of linear structure. 
For each $i\in\{1,\ldots,n\}$, $\lambda^{(i)}>0$ is a given constant.
Above, for each $v_i\in V_I$ and $\changed{(k)}\in E^+(i)$, $\alpha_{i,k}>0$ is a 
constant such that $\sum_{\changed{(k)}\in E^+(i)}\alpha_{i,k}=1$ holds, i.e.,
the coefficients $\alpha_{i,k}$ model how the flow splits at vertex $v_i$ into the 
flows along the edges from $E^+(i)$. This way, \eqref{eq:network_flow_conservation} conserves the flow.
We note that the theory can be extended to more general situations.
Exemplary, standard linear damping terms of type $\mu^{(i)}z^{(i)}(t,x)$ can be incorporated
in \eqref{eq:network_flow_PDE} for real constants $\mu^{(i)}>0$ for each $(i)\in E$
without any problem. Under additional assumptions, the coefficients $\lambda^{(i)}$ and
$\mu^{(i)}$ may also depend on time.
\changed{%
	Without loss of generality one could choose $\omega:=1$. 
	However, in order to clearly distinguish between temporal
	and spatial variables in notation, we stick to the seemingly more general
	situation where $\omega>0$ is arbitrary.
	Furthermore, the findings in this paper extend to connected networks without cycles,
	but apart from a more difficult notation, which also allows for vertices where flows are merged, 
	we do not believe that such a model comes along with a significantly different theory.
	We, thus, concentrate on tree-shaped networks.
}%

\subsubsection{Discussion of the hyperbolic conservation law}\label{sec:pde}

Let us first review a classical existence result for the linear hyperbolic conservation law \eqref{eq:network_flow}.
Therefore, we define a suitable control space by
\[
	C^1_{00}(\overline{\mathcal T})
	:=
	\left\{
		u\in C^1(\overline{\mathcal T})\,\middle|\,
		u(0)=0,\,u'(0)=0
	\right\}.
\]
We equip $C^1_{00}(\overline{\mathcal T})$ with the classical $C^1$-norm,
and note that this space is a closed subspace of $C^1(\overline{\mathcal T})$.
The proof of the following result, which is based on the method of characteristics, 
\changed{%
	can be distilled from \cite[Section~3.1, Theorems~3.4 and~3.6]{Bressan2000} 
	under the condition that we only consider positive velocities on the network 
	and, thus, all waves are moving with positive speed.
	}%

\begin{proposition}\label{prop:solution_of_conservation_law}
	For each $u\in C^1_{00}(\overline{\mathcal T})$, the hyperbolic conservation law \eqref{eq:network_flow}
	possesses a unique solution 
	$z:=(z^{(1)},\ldots,z^{(n)})\in C^1(\overline{\mathcal T\times\Omega},\R^n)$.
	The latter is explicitly given by
	\begin{equation}\label{eq:full_characterization_first_arc}
	\forall (t,x)\in\overline{\mathcal T\times\Omega}\colon\quad
	z^{(1)}(t,x)
	=
	\begin{cases}
		\frac{1}{\lambda^{(1)}}u(t-x/\lambda^{(1)})	&	t-x/\lambda^{(1)}>0,\\
		0											&	t-x/\lambda^{(1)}\leq 0
	\end{cases}
	\end{equation}
	on edge $(1)$, and for each $i\in\{1,\ldots,n\}$ such that $v_i\in V_I$ and $(k)\in E^+(i)$,
	we find
	\begin{equation}\label{eq:full_characterization_other_arcs}
	\forall (t,x)\in\overline{\mathcal T\times\Omega}\colon\quad
	z^{(k)}(t,x)
	=
	\begin{cases}
		\alpha_{i,k}\frac{\lambda^{(i)}}{\lambda^{(k)}}z^{(i)}(t-x/\lambda^{(k)},\omega)	
			&	t-x/\lambda^{(k)}>0,
		\\
		0																		
			&	t-x/\lambda^{(k)}\leq 0.
	\end{cases}
	\end{equation}
	Additionally, there is a constant $\kappa>0$, not depending on $u$, such that
	$\nnorm{z}_{C^1(\overline{\mathcal T\times\Omega},\R^n)}\leq\kappa\nnorm{u}_{C^1(\overline{\mathcal T})}$.
\end{proposition}

Let us note that formula \eqref{eq:full_characterization_other_arcs} can be used recursively to
determine the solution along all edges of the network. Indeed, based on \eqref{eq:full_characterization_first_arc},
the solution along all arcs from $E^+(1)$ can be computed. Next, using \eqref{eq:full_characterization_other_arcs},
it is possible to determine the flow along all edges starting in those vertices which are the end vertex of some 
edge in $E^+(1)$. Repeating this procedure, one can iterate through the whole network.

Clearly, \cref{prop:solution_of_conservation_law} justifies to introduce a map
from $C^1_{00}(\overline{\mathcal T})$ to 
$C^1(\overline{\mathcal T\times\Omega},\R^n)$ 
which assigns to each 
control function from $C^1_{00}(\overline{\mathcal T})$ the associated uniquely determined solution of
\eqref{eq:network_flow}.
This mapping is a linear operator which is continuous by \cref{prop:solution_of_conservation_law}.

Since we are interested in the optimal control of the system \eqref{eq:network_flow},
working with the control space $C^1_{00}(\overline{\mathcal T})$ induces some inherent
difficulties. First, this space is nonreflexive, i.e., to show the existence of
optimal solutions for optimization problems over \eqref{eq:network_flow} and the superordinate inverse
optimal control problem, which we state in \cref{sec:upper_level_problem}, would be challenging.
Second, the dual of this space, which naturally arises when using the adjoint approach
for the derivation of optimality conditions, is large and difficult to handle numerically.
It is, thus, a reasonable task to reconsider \eqref{eq:network_flow} from the
viewpoint of control functions $u\in L^2(\mathcal T)$.
\changed{%
	Besides, this choice allows for discontinuous controls which can be exploited 
	to model switches in the inflow.
	Observe that \eqref{eq:network_flow} does not need to possess a classical
	solution in the sense of \cref{prop:solution_of_conservation_law} anymore
	whenever the control function is not continuously differentiable. 
	To proceed, we follow \cite[Section~2.2]{Keimer2014}, see \cite[Section~2]{GugatKeimerLeugeringWang2015} as well,
	}%
to introduce a suitable weak formulation of \eqref{eq:network_flow} as stated below.
First, for the state $z^{(1)}$, we demand
\begin{equation}\label{eq:network_flow_weak_arc_1}
			\changed{\int_{\mathcal T_\tau}\int_\Omega}
				z^{(1)}(t,x)(\varphi_t(t,x)+\lambda^{(1)}\varphi_x(t,x))\changed{\dx\dt}
				=
				-\changed{\int_{\mathcal T_\tau}} u(t)\varphi(t,0)\dt
			\quad
			\forall\varphi\in \changed{W_\tau}
\end{equation}
\changed{%
	for all $\tau\in\overline{\mathcal T}$,
	where $\mathcal T_\tau:=(0,\tau)$ and
	}%
\[
	\changed{%
	W_\tau :=\left\{\varphi\in C^1(\overline{\mathcal T_\tau\times\Omega})\,\middle|\,
			\begin{aligned}
				\varphi(\cdot,\omega)&=0&&\text{ on }\overline{\mathcal T_\tau}
				\\
				\varphi(\tau,\cdot)&=0&&\text{ on }\overline\Omega
			\end{aligned}
		\right\}
	}%
\]
is the space of test functions.
Similarly as above, we demand
\begin{equation}\label{eq:network_flow_weak_other_arcs}
	\changed{\int_{\mathcal T_\tau}\int_\Omega}
				z^{(k)}(t,x)(\varphi_t(t,x)+\lambda^{(k)}\varphi_x(t,x))\changed{\dx\dt}
				=
				-
				\alpha_{i,k}\lambda^{(i)}
				\changed{\int_{\mathcal T_\tau}}z^{(i)}(t,\omega)\varphi(t,0)\dt
			\quad
			\forall\varphi\in \changed{W_\tau}
\end{equation}
\changed{%
	for all $\tau\in\overline{\mathcal T}$,
	}%
$v_i\in V_I$, and $(k)\in E^+(i)$.
\changed{%
	A function $z\in C(\overline\Omega,L^2(\mathcal T,\R^n))$ satisfying these requirements
	is referred to as a weak solution of the hyperbolic conservation law \eqref{eq:network_flow}.
	Recall that the function space $C(\overline\Omega,L^2(\mathcal T,\R^n))$ comprises all functions
	$z\colon\mathcal T\times\overline\Omega\to\R^n$ such that, for each $x\in\overline\Omega$,
	$z(\cdot,x)$ belongs to $L^2(\mathcal T,\R^n)$, and
	$\overline\Omega\ni x\to z(\cdot,x)\in L^2(\mathcal T,\R^n)$ is continuous.
}%
Let us emphasize that the boundary conditions 
\eqref{eq:network_flow_initials}, \eqref{eq:network_flow_inflow}
are incorporated in this alternative formulation of the dynamics also in weak sense
only (by definition of the space \changed{$W_\tau$}) 
since pointwise considerations are meaningless in Lebesgue spaces. 

The following result shows that the (classical) solution characterized 
in \cref{prop:solution_of_conservation_law} 
\changed{%
	(with controls chosen from $C^1_{00}(\overline{\mathcal T})$)
	also provides the uniquely determined weak solution of the hyperbolic conservation law \eqref{eq:network_flow}
	}%
if the control is chosen from $L^2(\mathcal T)$.

\begin{proposition}\label{prop:weak_solution_of_conservation_law}
	For each $u\in L^2(\mathcal T)$, the function
	\changed{%
		$z:=(z^{(1)},\ldots,z^{(n)})\in C(\overline\Omega,L^2(\mathcal T,\R^n))$
	}%
	characterized via \eqref{eq:full_characterization_first_arc}, \eqref{eq:full_characterization_other_arcs}
	\changed{%
		is the uniquely determined weak solution of
		the hyperbolic conservation law \eqref{eq:network_flow}.
	}%
	Additionally, there is a constant $\kappa>0$, not depending on $u$, such that
	\changed{%
		$\norm{z}_{C(\overline\Omega,L^2(\mathcal T,\R^n))}\leq\kappa\norm{u}_{L^2(\mathcal T)}$.
	}%
\end{proposition}
\begin{proof}
Let us start to show that $z^{(1)}$ given in \eqref{eq:full_characterization_first_arc}
satisfies \eqref{eq:network_flow_weak_arc_1} for 
\changed{each $\tau\in\overline{\mathcal T}$ and} given $u\in L^2(\changed{\mathcal T})$.
Therefore, we introduce a function
$\bar u\in L^2((-\omega/\lambda^{(1)},T))$ by
\[
	\forall t\in(-\omega/\lambda^{(1)},T)\colon\quad
	\bar u(t):=
	\begin{cases}
		u(t)	&	t>0,\\
		0		&	t\leq 0.
	\end{cases}
\]
Using a coordinate transformation with respect to the new domain
\[
	\changed{\Xi_\tau}
	:=
	\{
		(s,x)\in\R^2\,|\,x\in\Omega,\,s\in(-x/\lambda^{(1)},\changed{\tau}-x/\lambda^{(1)})
	\},
\]
we find, for each $\varphi\in \changed{W_\tau}$ and
$\bar\varphi(s,x):=\varphi(s+x/\lambda^{(1)},x)$ for all $(s,x)\in \changed{\Xi_\tau}$, the identities
\begin{align*}
	&\changed{%
		\int_{\mathcal T_\tau}\int_\Omega
			z^{(1)}(t,x)(\varphi_t(t,x)+\lambda^{(1)}\varphi_x(t,x))\dx\dt
			}%
	\\
	&\qquad 
		= \frac{1}{\lambda^{(1)}}\changed{\int_{\mathcal T_\tau}\int_\Omega}
		\bar u(t-x/\lambda^{(1)})(\varphi_t(t,x)+\lambda^{(1)}\varphi_x(t,x))\dx\dt\\
	&\qquad
		= \iint_{\changed{\Xi_\tau}}\bar u(s)\bar\varphi_x(s,x)\,\mathrm d(s,x)\\
	&\qquad
		= \int_{-\omega/\lambda^{(1)}}^{\changed{\tau}}
		\int_{\max(0,-\lambda^{(1)}s)}^{\min(\omega,\lambda^{(1)}(\changed{\tau}-s))}
			\bar u(s)\bar\varphi_x(s,x)\dx\ds\\
	&\qquad
		=\int_0^{\changed{\tau}} u(s)\int_0^{\min(\omega,\lambda^{(1)}(\changed{\tau}-s))}
			\bar\varphi_x(s,x)\dx\ds\\
	&\qquad
		=\int_0^{\changed{\tau}} u(s)(\bar\varphi(s,\min(\omega,\lambda^{(1)}(\changed{\tau}-s)))-\bar\varphi(s,0))\ds
			\\
	&\qquad
		=-\int_{\changed{\mathcal T_\tau}} u(s)\varphi(s,0)\ds.
\end{align*}
Above, we used 
\changed{%
	the fact that the determinant of the Jacobian associated with the chosen coordinate transform is $1$,
}%
the fundamental theorem of calculus, and
\[
	\bar\varphi(s,\min(\omega,\lambda^{(1)}(\changed{\tau}-s)))
	=
	\begin{cases}
		\varphi(\changed{\tau},\lambda^{(1)}(\changed{\tau}-s))=0	&	\text{if }\lambda^{(1)}(\changed{\tau}-s)<\omega,\\
		\varphi(s+\omega/\lambda^{(1)},\omega) =0	& 	\text{if } \lambda^{(1)}(\changed{\tau}-s)\geq\omega,
	\end{cases}
\]
which holds by definition of the space \changed{$W_\tau$} of test functions.
Thus, the function from \eqref{eq:full_characterization_first_arc} solves
\eqref{eq:network_flow_weak_arc_1}. 

\changed{%
	Next, we will verify $z^{(1)}\in C(\overline\Omega,L^2(\mathcal T))$.
	By definition of $z^{(1)}$, $z^{(1)}(\cdot,x)\in L^2(\mathcal T)$ is obvious for each $x\in\overline\Omega$.
	Let us pick $x_1,x_2\in\overline\Omega$ such that, without loss of generality,
	$x_1<x_2$.
	Then we have
	\begin{align*}
		&\nnorm{z^{(1)}(\cdot,x_1)-z^{(1)}(\cdot,x_2)}^2_{L^2(\mathcal T)}
		\\
		&\qquad
		=
		\frac{1}{(\lambda^{(1)})^2}\int_0^T\bigl(\bar u(t-x_1/\lambda^{(1)})-\bar u(t-x_2/\lambda^{(1)})\bigr)^2\dt
		\\
		&\qquad
		=\frac{1}{(\lambda^{(1)})^2}\left(
			\int_{x_1/\lambda^{(1)}}^{x_2/\lambda^{(1)}}u^2(t-x_1/\lambda^{(1)})\dt
			+
			\int_{x_2/\lambda^{(1)}}^{T}\bigl(u(t-x_1/\lambda^{(1)})-u(t-x_2/\lambda^{(1)})\bigr)^2\dt
			\right).
	\end{align*}
	As $|x_1-x_2|\to 0$, the first of these integrals trivially tends to $0$,
	and one can exploit similar arguments as used to prove
	\cite[Theorems~4.20, 4.21]{Dobrowolski2006}
	in order to verify that the second integral tends to $0$ as well.
	Hence, $x\mapsto z^{(1)}(\cdot,x)$ is continuous on $\overline\Omega$,
	i.e., $z^{(1)}\in C(\overline\Omega,L^2(\mathcal T))$.
}%
\changed{%
Let us also note that
\begin{align*}
	\nnorm{z^{(1)}}^2_{C(\overline\Omega,L^2(\mathcal T))}
	&=
	\max\limits_{x\in[0,\omega]}\nnorm{z^{(1)}(\cdot,x)}_{L^2(\mathcal T)}^2
	\\
	&=
	\frac{1}{(\lambda^{(1)})^2}\max\limits_{x\in[0,\omega]}\int_0^T\bar u^2(t-x/\lambda^{(1)})\dt
	\\
	&=
	\frac{1}{(\lambda^{(1)})^2}\max\limits_{x\in[0,\omega]}
	\int_{0}^{\max(0,T-x/\lambda^{(1)})}u^2(s)\ds
	\\
	&\leq 
	\frac{1}{(\lambda^{(1)})^2}\max\limits_{x\in[0,\omega]}\int_{0}^{T}u^2(s)\ds
	=
	\frac{1}{(\lambda^{(1)})^2}\norm{u}^2_{L^2(\mathcal T)},
\end{align*}
which gives
\[
	\nnorm{z^{(1)}}_{C(\overline\Omega,L^2(\mathcal T))}
	\leq
	(\lambda^{(1)})^{-1}\norm{u}_{L^2(\mathcal T)}.
\]
}%

The above arguments can be repeated in order to show that, for each $i\in\{1,\ldots,n\}$ such that
$v_i\in V_I$ and $(k)\in E^+(i)$, the function $z^{(k)}$ characterized in \eqref{eq:full_characterization_other_arcs}
satisfies \eqref{eq:network_flow_weak_other_arcs} 
\changed{%
for each $\tau\in\overline{\mathcal T}$, belongs to $C(\overline\Omega,L^2(\mathcal T))$,
and obeys the estimate
\[
	\nnorm{z^{(k)}}_{C(\overline\Omega,L^2(\mathcal T))}
	\leq
	\frac{\lambda^{(i)}}{\lambda^{(k)}}\nnorm{z^{(i)}}_{C(\overline\Omega,L^2(\mathcal T))}.
\]
}%

\changed{%
By iteration through the whole network, we obtain that the considered function $z$ is
a weak solution of the hyperbolic conservation law \eqref{eq:network_flow} which
belongs to $C(\overline\Omega,L^2(\mathcal T,\R^n))$.
Finally, uniqueness of the weak solution follows from \cite[Theorem~3.1.1]{Keimer2014}.
}%
\end{proof}

\Cref{prop:weak_solution_of_conservation_law} motivates the following definition.

\begin{definition}\label{def:solution_operator_conservation_law}
	Let $S\colon L^2(\mathcal T)\to \changed{C(\overline\Omega,L^2(\mathcal T,\R^n))}$ be the operator which assigns
	to each $u\in L^2(\mathcal T)$ the 
	\changed{%
		uniquely determined weak solution of the hyperbolic conservation law \eqref{eq:network_flow}
	}%
	which has been discussed in \cref{prop:weak_solution_of_conservation_law}. 
	For each $u\in L^2(\mathcal T)$, we make use of the notation $S^{(i)}(u)$ 
	\changed{%
		and $S^{(i)}_\omega(u)$
	}%
	to address the $i$-th component function of $S(u)$
	\changed{%
		and the function $S^{(i)}(u)(\cdot,\omega)\in L^2(\mathcal T)$
	}%
	for all $i\in\{1,\ldots,n\}$,
	\changed{
		respectively.
	}%
\end{definition}

The operator $S$ defined above encapsulates all the information provided by \eqref{eq:network_flow} and later on allows us to state
all optimization problems of interest in compact, state-reduced form.

\begin{remark}\label{rem:continuity_of_solution_operator}
	Due to \cref{prop:weak_solution_of_conservation_law}, the operator $S$ is linear and continuous.
	Furthermore, for each $i\in\{1,\ldots,n\}$, 
	\changed{%
		$S^{(i)}\colon L^2(\mathcal T)\to C(\overline\Omega,L^2(\mathcal T))$ 
		and
		$S^{(i)}_\omega\colon L^2(\mathcal T)\to L^2(\mathcal T)$
		are linear and continuous operators.
		}%
\end{remark}

\subsubsection{The lower-level problem and its solution operator}\label{sec:lower_level_solution}

Now, we are ready to formulate the optimal control problem of interest.
In general, the control function $u\in L^2(\mathcal T)$ has to be chosen 
\changed{%
	from a certain set of feasible controls
}%
such that a given demand is tracked while the control effort is
minimal. This is achieved by minimizing the expression
\begin{equation}\label{eq:lower_level_objective}
	f(u,\beta)
	:=
	\frac12\sum\limits_{(i)\in E_D}
		\norm{
			\changed{S^{(i)}_\omega}(u)-\mathsmaller\sum\nolimits_{\ell=1}^m\beta^{(i)}_\ell D^{(i)}_\ell
			}_{\changed{L^2(\mathcal T)}}^2
	+
	\frac{\sigma}{2}\norm{u}^2_{L^2(\mathcal T)}
\end{equation}
\changed{%
with respect to $u\in U_\textup{ad}\subset L^2(\mathcal T)$,
where
\begin{equation}\label{eq:Uad}
	U_\textup{ad}:=
	\left\{
		u\in L^2(\mathcal T)\,|\,
		u_\textup{a}(t)\leq u(t)\leq u_\textup{b}(t)\text{ a.e.\ on }\mathcal T
	\right\}
\end{equation}
is a standard box-constrained set induced by given measurable functions
$u_\textup{a}\colon\mathcal T\to\{-\infty\}\cup\R$ and $u_\textup{b}\colon\mathcal T\to\R\cup\{\infty\}$,
and we assume that $U_\textup{ad}$ is nonempty.
In \eqref{eq:lower_level_objective},
}%
$m\in\N$ such that $m\geq 2$ is a fixed natural number which does not depend on $(i)$ for simplicity.
Recall that $S(u)$ is used to represent the (weak) solution of \eqref{eq:network_flow} which we discussed in \cref{sec:pde},
$S^{(i)}(u)$ is the component of $S(u)$ which is associated with the edge $(i)\in E$, $i\in\{1,\ldots,n\}$,
\changed{%
and $S^{(i)}_\omega(u)\in L^2(\mathcal T)$ is obtained from $S^{(i)}(u)$ by fixing the spatial variable to $x:=\omega$,
see \cref{def:solution_operator_conservation_law} as well.
When controlling network flows, one typically tries to track the demand only
by means of the outflow at the demand vertices over time,
and this is reflected by definition of \eqref{eq:lower_level_objective}. 
The appearing regularization term aims to minimize the control effort 
and, in parallel, guarantees uniqueness of lower-level solutions 
as we will see later on, see \cref{prop:uniqueness_of_lower_level_solutions} below.
}%

In \eqref{eq:lower_level_objective}, for each $(i)\in E_D$, 
$D_1^{(i)},\ldots,D_m^{(i)}\in \changed{L^2(\mathcal T)}$ are typical demand profiles, 
i.e., suitable ansatz functions 
\changed{%
	modeling the outflow at demand vertices over time.
	}%
Furthermore, $\sigma>0$ is a regularization parameter.
The weights $\beta^{(i)}\in\Lambda^m$, where
\begin{equation}\label{eq:unit_simplex}
	\Lambda^m
	:=
	\left\{b\in\R^m\,\middle|\,b\geq 0,\,\mathsmaller\sum\nolimits_{\ell=1}^m b_\ell=1\right\}
\end{equation}
denotes the standard simplex in $\R^m$, are, at the lower-level stage, 
fixed and shall be reconstructed from measurements in a superordinate optimization problem, see \cref{sec:upper_level_problem}.
For fixed $(i)\in E_D$, $\sum_{\ell=1}^m\beta^{(i)}_\ell D^{(i)}_\ell$ 
plays the role of the desired demand, i.e., the outflow 
\changed{%
	at demand vertex $v_i$ over time
}%
we are seeking.
In this regard and by definition of $\Lambda^m$, $D^{(i)}_1,\ldots,D^{(i)}_m$ 
play the role of extreme cases for the desired demand 
as $\sum_{\ell=1}^m\beta^{(i)}_\ell D^{(i)}_\ell$ is a convex combination of these functions. 
If $D^{(i)}_1,\ldots,D_m^{(i)}$ are interpreted merely as basis functions of a suitable finite-dimensional
subspace of \changed{$L^2(\mathcal T)$}, where the desired demand has to be chosen from, then the weights $\beta^{(i)}$ can be taken as
arbitrary vectors from $\R^m$, resulting in an easier superordinate reconstruction problem.
In \eqref{eq:lower_level_objective}, we used $\beta:=(\beta^{(i)})_{(i)\in E_D}\in(\R^m)^{|V_D|}$ for brevity of notation.

We investigate the lower-level optimal control problem
\begin{equation}\label{eq:lower_level}\tag{LL$(\beta)$}
	\min\limits_{u}\{f(u,\beta)\,|\,u\in \changed{U_\textup{ad}}\}.
\end{equation}
Here, $f\colon L^2(\mathcal T)\times (\R^m)^{|V_D|} \to\R$
is the function defined in \eqref{eq:lower_level_objective},
\changed{%
	and the set of feasible controls $U_\textup{ad}$ is given in \eqref{eq:Uad}.
}%
Observing that the operators \changed{$S^{(i)}_\omega\colon L^2(\mathcal T)\to L^2(\mathcal T)$}, 
$(i)\in E_D$, are linear and continuous while keeping the
presence of the nonvanishing regularization term $\tfrac{\sigma}{2}\norm{u}^2_{L^2(\mathcal T)}$ in mind, the function 
$f(\cdot,\beta)$ is continuous and \changed{uniformly} convex for each $\beta\in(\R^m)^{|V_D|}$. 
\changed{%
	Furthermore, $U_\textup{ad}$ is, by construction, convex and closed.
}%
Thus, noting that $L^2(\mathcal T)$ is a reflexive Banach space, 
\eqref{eq:lower_level} possesses a uniquely determined \changed{(global)} minimizer,
\changed{%
see e.g.\ \cite[Theorem~2.16]{Troeltzsch2009} for a similar result.
}%

\begin{proposition}\label{prop:uniqueness_of_lower_level_solutions}
	For each $\beta\in(\R^m)^{|V_D|}$, \eqref{eq:lower_level} possesses a \changed{uniquely determined (global) minimizer}.
\end{proposition}

The above result motivates the following definition.
\begin{definition}\label{def:lower_level_solution_operator}
	Let $\Psi\colon(\R^m)^{|V_D|}\to L^2(\mathcal T)$ be the operator which assigns to each
	$\beta\in(\R^m)^{|V_D|}$ the uniquely determined \changed{(global) minimizer} of \eqref{eq:lower_level}.
\end{definition}

\changed{%
In the subsequent result, we show via standard arguments that the mapping $\Psi$ defined in \cref{def:lower_level_solution_operator} is Lipschitz continuous.
}%

\begin{proposition}\label{prop:lower_level_solution_operator}
	The mapping $\Psi$ is
	\changed{%
		Lipschitz continuous.
		}%
\end{proposition}
\begin{proof}
\changed{%
	For $\iota=1,2$, we choose $\beta_\iota\in(\R^m)^{|V_D|}$ and set $\bar u_\iota:=\Psi(\beta_\iota)$.
	Convexity and continuous Fr\'{e}chet differentiability of $f(\cdot,\beta_\iota)$ 
	as well as convexity and closedness of $U_\textup{ad}$ yield
	that $\bar u_\iota$ is the global minimizer of
	\hyperref[eq:lower_level]{(LL$(\beta_\iota)$)} if and only if the condition
	\begin{equation}\label{eq:lower_level_optimality}
		\forall u\in U_\textup{ad}\colon\quad
		\innerprod{f'(\bar u_\iota,\beta_\iota)}{u-\bar u_\iota}_{L^2(\mathcal T)}\geq 0
	\end{equation}
	is valid, see \cite[Lemma~2.21]{Troeltzsch2009}. Above, 
	$\innerprod{\cdot}{\cdot}_{L^2(\mathcal T)}\colon L^2(\mathcal T)\times L^2(\mathcal T)\to\R$
	denotes the standard inner product of the Hilbert space $L^2(\mathcal T)$.
}%

	Applying the chain rule, 
	we find an explicit formula for the derivative of $f(\cdot,\beta_\iota)$ with respect to $u$.
	More precisely, we have 
	\begin{equation}\label{eq:derivative_of_f}
		f'_u(u,\beta_\iota)=\mathcal A(u)-\mathcal B(\beta_\iota)
	\end{equation} 
	for the continuous linear operators
	$\mathcal A\colon L^2(\mathcal T)\to L^2(\mathcal T)$ 
	and $\mathcal B\colon (\R^m)^{|V_D|}\to L^2(\mathcal T)$ given by
	\[
		\forall u\in L^2(\mathcal T)\colon\quad
		\mathcal A(u)
		:=
		\sum_{(i)\in E_D}(\changed{(S^{(i)}_\omega)^*\circ S^{(i)}_\omega})(u)+\sigma u,
	\]
	where, for each $(i)\in E_D$, 
	\changed{$(S^{(i)}_\omega)^*\colon L^2(\mathcal T)\to L^2(\mathcal T)$}
	is the adjoint of \changed{$S^{(i)}_\omega\colon L^2(\mathcal T)\to L^2(\mathcal T)$},
	which is a linear, continuous operator again, 
	see \cref{rem:continuity_of_solution_operator}, and
	\[
		\forall \beta\in(\R^m)^{|V_D|}\colon\quad
		\mathcal B(\beta)
		:=
		\sum_{(i)\in E_D}\sum_{\ell=1}^m
			\beta^{(i)}_\ell\changed{(S^{(i)}_\omega)^*}(D^{(i)}_\ell).
	\]

\changed{%
	Hence, from \eqref{eq:lower_level_optimality} and \eqref{eq:derivative_of_f}, we find
	\begin{align*}
		\innerprod{\mathcal A(\bar u_1)-\mathcal B(\beta_1)}{\bar u_2-\bar u_1}_{L^2(\mathcal T)}
		&\geq 0,\\
		\innerprod{\mathcal A(\bar u_2)-\mathcal B(\beta_2)}{\bar u_1-\bar u_2}_{L^2(\mathcal T)}
		&\geq 0.
	\end{align*}
	Adding up these inequalities and performing some rearrangements yields
	\begin{equation}\label{eq:standard_estimate}
		\innerprod{\mathcal A(\bar u_1-\bar u_2)}{\bar u_1-\bar u_2}_{L^2(\mathcal T)}
		\leq
		\innerprod{\mathcal B(\beta_1-\beta_2)}{\bar u_1-\bar u_2}_{L^2(\mathcal T)}.
	\end{equation}
	By definition of $\mathcal A$, we find
	\begin{align*}
		&\innerprod{\mathcal A(\bar u_1-\bar u_2)}{\bar u_1-\bar u_2}_{L^2(\mathcal T)}
		\\
		&\qquad
		=
		\sum_{(i)\in E_D}
			\innerprod{((S^{(i)}_\omega)^*\circ S^{(i)}_\omega)(\bar u_1-\bar u_2)}
			{\bar u_1-\bar u_2}_{L^2(\mathcal T)}
		+
		\sigma \nnorm{\bar u_1-\bar u_2}^2_{L^2(\mathcal T)}
		\\
		&\qquad=
		\sum_{(i)\in E_D}\nnorm{S^{(i)}_\omega(\bar u_1-\bar u_2)}_{L^2(\mathcal T)}^2
		+
		\sigma \nnorm{\bar u_1-\bar u_2}^2_{L^2(\mathcal T)}
		\\
		&\qquad\geq
		\sigma \nnorm{\bar u_1-\bar u_2}^2_{L^2(\mathcal T)},
	\end{align*}	 
	and due to the continuity of $\mathcal B$, there is a constant $\kappa>0$,
	not depending on $\beta_1$ and $\beta_2$, such that
	\begin{align*}
		\innerprod{\mathcal B(\beta_1-\beta_2)}{\bar u_1-\bar u_2}_{L^2(\mathcal T)}
		&\leq
		\nnorm{\mathcal B(\beta_1-\beta_2)}_{L^2(\mathcal T)}
		\nnorm{\bar u_1-\bar u_2}_{L^2(\mathcal T)}
		\\
		&\leq
		\kappa\nnorm{\beta_1-\beta_2}_{(\R^m)^{|V_D|}}
		\nnorm{\bar u_1-\bar u_2}_{L^2(\mathcal T)}.
	\end{align*}
	Combining this with \eqref{eq:standard_estimate}, we end up with
	\[
		\nnorm{\bar u_1-\bar u_2}_{L^2(\mathcal T)}
		\leq
		(\kappa/\sigma)\nnorm{\beta_1-\beta_2}_{(\R^m)^{|V_D|}},
	\]
	which shows the desired Lipschitzness of $\Psi$.
}%
\end{proof}

\changed{%
	Let us close this subsection with some remarks.
}%

\begin{remark}\label{rem:box_constraints_at_lower_level}
\begin{enumerate}
\item
	\changed{%
	Let us mention that whenever the box constraints 
	in the optimization problem \eqref{eq:lower_level} are dropped,
	i.e., $u_\textup{a}\equiv-\infty$ and $u_\textup{b}\equiv\infty$ in \eqref{eq:Uad}, 
	then the associated solution operator $\Psi$ is linear and continuous
	(as the operator $\mathcal A$, 
	defined in the proof of \cref{prop:lower_level_solution_operator},
	is coercive and, thus, as a consequence of the Lax--Milgram lemma,
	see e.g.\ \cite[Theorem~2.29]{Dobrowolski2006}, continuously invertible).
	}%
\item 
For the theory in this paper, 
it is essential that the dynamics in \eqref{eq:lower_level} are linear 
as this, together with the structure of the objective function, guarantees that, 
for fixed $\beta\in(\R^m)^{|V_D|}$, \eqref{eq:lower_level} is a convex optimization problem which
can be solved to (global) optimality without any difficulty. 
\changed{%
	Recall that solving the lower-level problem globally is an intrinsic assumption
	in bilevel optimization.
}%
In case where \eqref{eq:network_flow} is replaced by a nonlinear
hyperbolic conservation law, it first has to be studied whether \eqref{eq:lower_level} actually
possesses a global solution, 
\changed{%
	which can be done in line with the findings in \cite{Keimer2014}.
}%
The far more challenging task is to find a procedure which reliably computes the global minimizer
of \eqref{eq:lower_level}, which is a nontrivial problem in the presence of nonlinear dynamics as
additional local minimizers and stationary points, which are not even local minimizers, may exist. 
\end{enumerate}
\end{remark}

\subsection{The upper-level problem}\label{sec:upper_level_problem}

To motivate the reconstruction problem,
we first consider the optimal control problem
\begin{equation}\label{eq:forward_transportation_problem}
	\min\limits_{z,u}\{g(z,u)\,|\,z=S(u),\,u\in \changed{U_\textup{ad}}\}
\end{equation}
where $g\colon \changed{C(\overline\Omega,L^2(\mathcal T,\R^n))}\times L^2(\mathcal T)\to\R$ 
is a classical tracking-type function given by
\[
	g(z,u)
	:=
	\frac12\sum\limits_{(i)\in E_D}
		\nnorm{
			\changed{z^{(i)}(\cdot,\omega)}-D^{(i)}_\textup{d}
			}_{\changed{{L^2(\mathcal T)}}}^2
	+
	\frac{\sigma}{2}\norm{u}^2_{{L^2(\mathcal T)}}
\]
for each $z\in \changed{C(\overline\Omega,L^2(\mathcal T,\R^n))}$ and $u\in L^2(\mathcal T)$,
\changed{%
	and $U_\textup{ad}$ is the set of feasible controls defined in \eqref{eq:Uad}.
}%
In \eqref{eq:forward_transportation_problem}, 
$S$ denotes the solution operator associated with the hyperbolic conservation law
\eqref{eq:network_flow}, see \cref{def:solution_operator_conservation_law}.
Furthermore, $D^{(i)}_\textup{d}\in \changed{L^2(\mathcal T)}$, $(i)\in E_D$, is some desired
demand \changed{at vertex $v_i$} which shall be approximated by the associated components
of the network flow resulting from a suitable choice of the control function $u$.

We assume that, e.g., by numerical experiments, $p\in\N$ (approximate) pairs of solutions 
$(z_{\textup{o},r},u_{\textup{o},r})\in \changed{C(\overline\Omega,L^2(\mathcal T,\R^n))}\times L^2(\mathcal T)$,
$r=1,\ldots,p$, of \eqref{eq:forward_transportation_problem} have been obtained
for noisy desired demands, or that the measurement itself has been inexact so that
$(z_{\textup{o},r},u_{\textup{o},r})$ on their own are noisy.
Let us emphasize that $z_{\textup{o},r}$ is some (noisy) state corresponding to $u_{\textup{o},r}$ for
each $r=1,\ldots,p$, i.e., $z_{\textup{o},r}\approx S(u_{\textup{o},r})$.
Our goal is to reconstruct 
the functions $D^{(i)}_\textup{d}$, $(i)\in E_D$, from these observations.
In order to do so, we presume that, for suitably chosen $\beta^{(i)}\in \Lambda^m$, where
$\Lambda^m$ has been defined in \eqref{eq:unit_simplex}, we can ensure
\[
	D^{(i)}_\textup{d}
	\approx
	\sum_{\ell=1}^m\beta^{(i)}_\ell D^{(i)}_\ell
\]
for all $(i)\in E_D$, where $D^{(i)}_1,\ldots,D^{(i)}_m$ are the available 
prototypical demand profiles we already mentioned in \cref{sec:lower_level_solution}.
\changed{%
	Hence, \eqref{eq:lower_level} can be interpreted as a special
	instance of \eqref{eq:forward_transportation_problem}.
}%

In this regard, the reconstruction task can be modeled via the bilevel optimization
problem
\begin{equation}\label{eq:upper_level_problem}\tag{UL}
	\min\limits_{\beta,u}\left\{
		\frac12\sum_{r=1}^p\left(
			\norm{\mathcal C(S(u)-z_{\textup{o},r})}_{H_z}^2
			+
			\norm{\mathcal D(u-u_{\textup{o},r})}_{H_u}^2
			\right)
		\,\middle|\,
			\begin{aligned}
				&\beta\in(\Lambda^m)^{|V_D|}\\
				&u=\Psi(\beta)
			\end{aligned}
	\right\}.
\end{equation}
Here, some linear, continuous mapping 
$\mathcal C\colon \changed{C(\overline\Omega,L^2(\mathcal T,\R^n))}\to H_z$ 
plays the role of an observation operator that maps the difference of the network flows 
and the observed flows to some observation space $H_z$, which is assumed to be Hilbert.
Similarly, the linear, continuous mapping $\mathcal D\colon L^2(\mathcal T)\to H_u$ 
plays the role of an observation operator
addressing the inflow, and the observation space $H_u$ is a Hilbert space.
Let us also recall that $\Psi$ denotes the solution operator associated with the
lower-level problem \eqref{eq:lower_level}, see \cref{def:lower_level_solution_operator}.
As already pointed out  in \cref{sec:lower_level_solution}, it might also be reasonable to
remove the constraint $\beta\in(\Lambda^m)^{|V_D|}$ from \eqref{eq:upper_level_problem},
resulting in a simpler problem.

We would like to mention a prototypical choice for the operator $\mathcal C$ here.
It seems to be a reasonable idea to inspect the behavior of the flow along certain parts of (selected) edges pointing
towards demand vertices from $V_D$. Thus, for some nonempty set $E'\subset E_D$ and 
\changed{%
	(potentially trivial) closed intervals
	$I^{(i)}\subset\overline\Omega$, $(i)\in E'$, one could exploit
	$\mathcal C\colon C(\overline\Omega,L^2(\mathcal T,\R^n))\to\prod_{(i)\in E'}C(I^{(i)},L^2(\mathcal T))$
given by
\[
	\forall z\in C(\overline\Omega,L^2(\mathcal T,\R^n))\colon\quad
	\mathcal C(z)
	:=
	\bigl(z^{(i)}(\cdot,\cdot)|_{\mathcal T\times I^{(i)}}\bigr)_{(i)\in E'}.
\]
Particularly, one can choose $I^{(i)}:=\{\omega\}$ for each $(i)\in E'$ in order to compare
network flows and observed flows merely by means of the outflow at certain demand vertices.
}%
Note that the operator $\mathcal C$ might be also restricted in time, 
see \cref{sec:numerical_treatment} 
and our numerical experiments in \cref{sec: simulations7}.
The operator $\mathcal D$ may be chosen similarly, 
restricting the observation of the inflow in time,
\changed{%
	including the trivial case where $\mathcal D$ is the identity.
}%
Whenever $\mathcal D$ is chosen to be the zero operator, the inflow is not observed.

Let us demonstrate that the reconstruction problem \eqref{eq:upper_level_problem} possesses an
optimal solution.
\begin{proposition}\label{thm:existence}
	The optimization problem \eqref{eq:upper_level_problem} possesses a
	\changed{globally} optimal solution.
\end{proposition}
\begin{proof}
	We note that \eqref{eq:upper_level_problem} can be transferred into a
	finite-dimensional optimization problem by plugging the lower-level
	solution operator $\Psi$ into the objective function.
	\changed{%
		It is obvious that a point $\beta\in(\R^m)^{|V_D|}$
		is a global minimizer of the resulting control-reduced problem
		if and only if $(\beta,\Psi(\beta))$ is a global minimizer
		of \eqref{eq:upper_level_problem}.
	}%
	By continuity of $\Psi$, see \cref{prop:lower_level_solution_operator}, 
	and continuity of $\mathcal C$ as well as $\mathcal D$, 
	the objective function of the reduced
	problem is then continuous, while its feasible set $(\Lambda^m)^{|V_D|}$
	is nonempty and compact. Thus, the reduced problem possesses a global
	minimizer $\bar\beta\in(\R^m)^{|V_D|}$ by the Weierstraß theorem, and this yields that
	$(\bar\beta,\Psi(\bar\beta))$ solves \eqref{eq:upper_level_problem}
	to global optimality.
\end{proof}

\changed{%
	Although being globally Lipschitz continuous, see \cref{prop:lower_level_solution_operator},
	the lower-level solution operator $\Psi$ is likely to be nonsmooth apart from the
	special situation where no control constraints are present,
	see \cref{rem:box_constraints_at_lower_level}.
	Eliminating the control variable $u$ in \eqref{eq:upper_level_problem} by
	plugging $\Psi$ into the objective function, thus, leads to a finite-dimensional
	but nonconvex, nonsmooth optimization problem with polyhedral constraints.
	Whenever $U_\textup{ad}=L^2(\mathcal T)$ holds, $\Psi$ is linear,
	see \cref{rem:box_constraints_at_lower_level} again, and
}%
\eqref{eq:upper_level_problem} is actually a convex optimization problem. 
In this particular situation, numerical methods
which identify stationary points of \eqref{eq:upper_level_problem} may already
compute global minimizers of the problem. This is a rare property in
hierarchical optimization where the multilevel structure is, typically, a
source of nonconvexity and nonsmoothness, 
\changed{%
	and this problem we also face in the general setting where
	control constraints are present.
	}%

\section{\changed{Numerical solution and computational results}}\label{sec:experiments}

\changed{%
	In this section, we first describe how \eqref{eq:upper_level_problem} can be solved
	in numerical practice. Second, results of some computational experiments are presented.
}%

\subsection{\changed{Numerical solution of the problem}}\label{sec:numerical_treatment}

For the network discretization, 
\changed{%
we choose a time grid $(t_j)_{j=1}^{J}$   
of $J\in\N$ discretization points such that $t_j:= (j-1)\Delta t$ for all $j\in\{1,\ldots,J\}$,
where $\Delta t>0$ is a given temporal stepsize, 
and a spatial discretization of each edge $(i)$, 
represented by the interval $(0,\omega)$, as $(x^{(i)}_q)_{q=1}^{L^{(i)}}$,
where $L^{(i)}\in\N$ is the number of discretization points,
$x^{(i)}_q := (q-1)\Delta x^{(i)}$ for all $q\in\{1,\ldots,L^{(i)}\}$, 
and $\Delta x^{(i)}>0$ is the spatial stepsize for edge $(i)$.
}%
The transported quantities $z^{(i)}_{j,q}$ at time $t_j$ and position \changed{$x_q^{(i)}$}
given by the PDE in \eqref{eq:network_flow_PDE} are calculated using a left-sided 
upwind scheme, i.e.,
\begin{align}
	\label{eq:upwindScheme}
	z^{(i)}_{j,q} 
	= 
	z_{j-1,q}^{(i)} - \frac{\Delta t}{\Delta x^{(i)}} \lambda^{(i)} 
		\left(z_{j-1,q}^{(i)} - z_{j-1,q-1}^{(i)} \right),
	\quad j\in\{2,\ldots,J\},\,q\in\{2,\ldots,L^{(i)}\}.
\end{align}
We also note that \eqref{eq:network_flow_initials} translates into $z_{1,q}^{(i)}=0$ for
all $q\in\{1,\ldots,L^{(i)}\}$.
At the junctions, according to \eqref{eq:network_flow_inflow} 
and \eqref{eq:network_flow_conservation}, we require 
\begin{align}
	\label{eq:upwindJunctions}
	z^{(1)}_{j,1} 
	= 
	\frac{u_j}{\lambda^{(1)}},
	\qquad
	z^{(k)}_{j,1} 
	= 
	\alpha_{i,k}\frac{\lambda^{(i)}}{\lambda^{(k)}}z^{(i)}_{j,L^{(i)}},\quad
	\changed{v_i\in V_I,\,(k)} \in E^+(i),\,j\in\{1,\ldots,J\}
\end{align}
where $u_j:=u(t_j)$ for all $j\in\{1,\ldots,J\}$.
For $\frac{\Delta t}{\Delta x^{(i)}} \lambda^{(i)} = 1$, the upwind scheme shows no diffusion.
Therefore, we set $\Delta x^{(i)} := \lambda^{(i)}\Delta t$ 
which leads to different spatial grids on the different edges
whenever the respective coefficients $\lambda^{(i)}$ are not the same.

We use this discretization for a finite differences approximation of the lower-level problem 
\eqref{eq:lower_level}.
We define $S^{(i),{L^{(i)}}}\in\R^{J\times J}$
\changed{%
	to be the (discrete) realization of $S^{(i)}_\omega$ 
}%
such that $\sum_{\nu=1}^JS^{(i),{L^{(i)}}}_{j,\nu}u_\nu$ approximates the influence of the discretized  inflow
on the density $z^{(i)}_{j,{L^{(i)}}}$ at time $t_j$ and spatial point $\omega$. 
\changed{%
Further, we denote the discrete versions of the demand profiles 
$D_1^{(i)},\ldots,D_m^{(i)}$ for edge $(i)\in E_D$ by 
$\tilde D^{(i)}_1,\ldots,\tilde D^{(i)}_m\in\R^J$.
}%

\changed{%
	For our computations, we will exploit that the columns of $S^{(i),L^{(i)}}$ are orthogonal to each other. This is the case since, due to the special structure of the PDEs, 
	there is a one-to-one correspondence between the inflow into the system and the outflow out of the system. 
	Therefore, in the discretized setting, there is a unique time point for the inflow
	that determines the outflow at the corresponding outflow time. 
	This property enforces the matrix $S^{(i),L^{(i)}}$ to be nonzero on its subdiagonal. 
	Consequently, $S^{(i),L^{(i)}}$ is orthogonal.
}%

For a given convex combination of base demands by the vector $\beta$ 
and using \eqref{eq:derivative_of_f}, we obtain the optimal inflow 
in the discretized setting 
\changed{in the absence of control constraints} 
when solving the linear system $Au - B\beta = 0$ where $A$ is given by 
\[
	A
	:=
	\sum_{(i)\in E_D}\left(S^{(i),{L^{(i)}}}\right)^\top\left(S^{(i),L^{(i)}}\right)
	+
	\sigma\mathbb I_J,
\]
where $\mathbb I_J\in\R^{J\times J}$ is the identity matrix,
and
\[
	B
	:=
	\begin{bmatrix}
		Q^{(i)}
	\end{bmatrix}_{(i)\in E_D}.
\]
Above, for each $(i)\in E_D$, $Q^{(i)}\in\R^{J\times m}$ is given by
\[
	Q^{(i)}
	:=
	\begin{bmatrix}
		\left(S^{(i),{L^{(i)}}}\right)^\top\tilde D^{(i)}_1 
		& \ldots & 
		\left(S^{(i),L^{(i)}}\right)^\top \tilde D^{(i)}_m
	\end{bmatrix}.
\]
\changed{%
We note that the discretized lower-level problem is equivalent to
\[
	\min\limits_{u}\{\tfrac12 u^\top Au-(B\beta)^\top u\,|\,
		u_\textup{a}\leq u\leq u_\textup{b}\},
\]
where $u_{\textup{a},j}:=u_\textup{a}(t_j)$ and $u_{\textup{b},j}:=u_\textup{b}(t_j)$
for all $j=1,\ldots,J$.
We obtain the solution of this problem by projecting
the solution of the linear equation $Au-B\beta=0$ onto the feasible box, since $A$ is a diagonal
positive definite matrix by orthogonality of $S^{(i),L^{(i)}}$, $(i)\in E_D$, 
see \cref{sec:special_QPs} for details.
}%

For the upper-level problem \eqref{eq:upper_level_problem}, 
we apply the same discretization technique with different stepsizes, 
see \cref{sec:documentation}, and consider, 
if not specified differently, the observation operator $\mathcal{C}$ in which we only observe the densities 
\changed{%
	at the demand vertices from $V_D$, corresponding to the last discretization points of the edges in $E_D$,
	as well as at the first discretization point of edge $(1)$, monitoring the inflow at $v_0$.
	Additionally, $\mathcal D$ is the zero operator in our experiments.
	}%
Further details and some numerical examples are explained in \cref{sec:documentation} 
where it is also described how $\mathcal{C}$ and $\mathcal D$ can be adjusted. 

\changed{%
Inserting the discretized solution operator of the lower-level problem
into the objective function of the discretized upper-level problem results
in a nonsmooth optimization problem with affine constraints, 
and we solve the latter using MATLAB's \texttt{patternsearch} solver in default mode. 
We want to emphasize that the performance of this optimization routine heavily depends 
on the initial point that is handed over to the solver.
This, however, is not surprising as the considered nonsmooth problem of interest is nonconvex
and, thus, likely to possess several local minimizers and stationary points which are
different from its global minimizers.
As the model is designed to reconstruct certain reference parameters from noisy data,
we initialize \texttt{patternsearch} with a perturbed version of these reference parameters
to face this problem.
We note that, in the absence of lower-level control constraints,
the resulting single-level problem is a simple convex quadratic problem which can
be solved, exemplary, with the aid of MATLAB's \texttt{quadprog} routine,
and the aforementioned issues do not occur.
}%

\subsection{\changed{General set-up of experiments}}\label{sec:general_set_up}

We consider the tree-shaped network presented in \cref{fig:example_network_large} in which each edge has a length of $\omega=1$. 
\begin{figure}[ht]
	\begin{center}
		\begin{tikzpicture}[-,>=stealth',shorten >=1pt, semithick]
			
			\node[bnode,label=left:$v_0$] at (-2,0) (0){};
			\node[bnode,label=above:$v_1$] at (0,0) (1){};
			\node[bnode,label=above:$v_2$] at (2,1) (2){};
			\node[bnode,label=below:$v_3$] at (2,-1) (3){};
			\node[bnode,label=above:$v_4$] at (4,2) (4){};
			\node[bnode,label=above:$v_5$] at (4,0) (5){};
			\node[bnode,label=below:$v_6$] at (4,-2) (6){};
			\node[bnode,label=right:$v_7$] at (6,2) (7){};
			\node[bnode,label=right:$v_8$] at (6,1) (8){};
			\node[bnode,label=right:$v_9$] at (6,0) (9){};
			\node[bnode,label=right:$v_{10}$] at (6,-1) (10){};
			\node[bnode,label=right:$v_{11}$] at (6,-2) (11){};
			
			\path[->] (0) edge node[above]{$(1)$} (1);
			\path[->] (1) edge node[above left]{$(2)$} (2);
			\path[->] (1) edge node[below left]{$(3)$} (3);
			\path[->] (2) edge node[above]{$(4)$} (4);
			\path[->] (2) edge node[below]{$(5)$} (5);
			\path[->] (3) edge node[below]{$(6)$} (6);
			\path[->] (4) edge node[above]{$(7)$} (7);
			\path[->] (4) edge node[below]{$(8)$} (8);
			\path[->] (5) edge node[above]{$(9)$} (9);
			\path[->] (6) edge node[above]{$(10)$} (10);
			\path[->] (6) edge node[below]{$(11)$} (11);
		\end{tikzpicture}
	\end{center}
	\caption{%
		The network considered in \cref{sec:documentation} with $V_D=\{v_7,v_8,v_9,v_{10},v_{11}\}$, 
		$V_I=\{v_1,v_2,v_3,v_4,v_5,v_6\}$, 
		and $E_D=\{(7),(8),(9),(10),(11)\}$.
		}%
	\label{fig:example_network_large}
\end{figure}
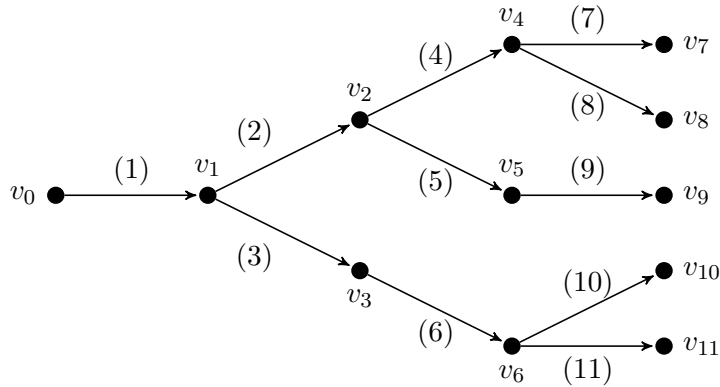
The velocities are chosen identically for all edges, we use $\lambda^{(i)} = 10$, $i=1,\dots,11$. 
The stepsizes are given by  $\Delta t = \frac{1}{60}$, $\Delta x = \frac{1}{6}$ for the backward calculation 
and \changed{$\Delta t = \frac{1}{70}$, $\Delta x = \frac{1}{7}$} for the forward calculation, which are chosen differently to avoid an inverse crime, 
see \cite[page~154]{ColtonKress2013}, 
\changed{for the unconstrained examples, 
and $\Delta t = \frac{1}{20}$, $\Delta x = \frac{1}{2}$ (backward calculation), $\Delta t = \frac{1}{30}$, $\Delta x = \frac{1}{3}$ (forward calculation) 
when applying constraints to the inflow in order keep reasonable running times.} Note that $\Delta x^{(i)}=\Delta x$ is exploited, $i=1,\ldots,11$.
In both cases, the Courant--Friedrichs--Lewy condition holds true with equality, i.e., $\frac{\Delta t}{\Delta x}\lambda^{(i)} = 1$, $i=1,\dots,11$,
to avoid diffusion in the numerical scheme. 
The distribution parameters are set to 
\[
	\begin{aligned}
		\alpha_{1,2}&=0.65,&\quad
		\alpha_{2,4}&=0.7,&\quad
		\alpha_{4,7}&=0.5,&\quad
		\alpha_{6,10} &= 0.4,&
		\\		
		\alpha_{1,3}&=0.35,& 
		\alpha_{2,5}&=0.3,&
		\alpha_{4,8}&=0.5,&
		\alpha_{6,11} &= 0.6.&
	\end{aligned}
\]
We consider the evolution of the demand within one week, i.e., $T=168$
where one time unit represents one hour and assume four underlying base demand levels which are visualized in \cref{fig: baseDemands}
and chosen as
\begin{itemize}
	\item a time constant level of the demand: \\
		$D_1(t) = 4$,
	\item a daily varying level at which we attain the highest level in the morning:\\
		$D_2(t) = 2 + \sin\left(\pi(t-2)/12\right)$,
	\item a daily varying level at which we attain the highest level in the afternoon:\\ 
		$D_3(t) = 2 + \sin\left(\pi(t-10)/12\right)$,
	\item a level that illustrates the lower demand during the weekend:\\
		$D_4(t) = \mathbbm{1}_{[0,120]}(t)$.
\end{itemize}
\begin{figure}
	\centering
	\includegraphics[scale=0.6]{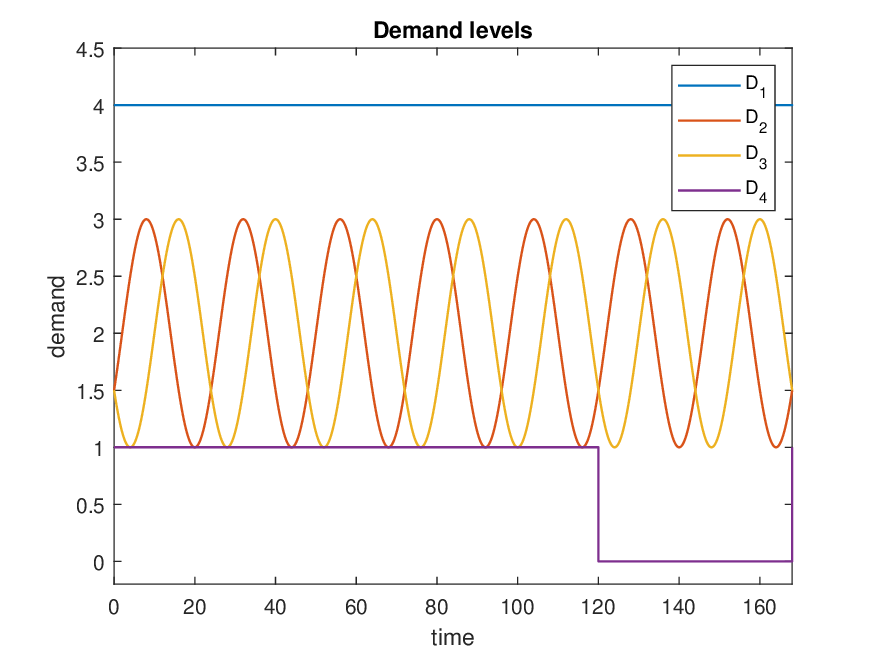}
	\caption{Illustration of the four base demand levels.}
	\label{fig: baseDemands}
\end{figure}
These choices can similarly be found for example for the electricity market in \cite{CoskunKorn2021} 
and describe the identified two-peak pattern of demand in the intraday market ($D_2, D_3$) 
as well as the phenomenon referred to as the \textit{weekend effect} ($D_4$).
For the prototypical demand profiles, we make use of
$D^{(i)}_\ell:=\hat d^{(i)}\,D_\ell$, $i\in\{7,\ldots,11\}$, $\ell\in\{1,2,3,4\}$,
where
\[
	\hat{d}^{(7)}=0.2275,
	\quad
	\hat{d}^{(8)}=0.2275,
	\quad
	\hat{d}^{(9)}=0.195,
	\quad
	\hat{d}^{(10)}=0.14,
	\quad
	\hat{d}^{(11)}=0.21.
\]
This choice proportionally accounts for the different distribution parameters in the network.
The historical observations are basically generated using the initial weights 
\begin{align}
	\label{initialBeta}
	(\beta_1, \beta_2, \beta_3,\beta_4) = (0.2,0.15,0.2,0.45).
\end{align}
In every time step and for every demand vertex, the base demand levels are perturbed by random variables 
\[
	Z_1^{(i)}\sim \mathcal{N}(0,1),\quad
	Z_2^{(i)}\sim \mathcal{N}(0,1/4),\quad
	Z_3^{(i)}\sim \mathcal{N}(0,1/4),\quad
	Z_4^{(i)}\sim \mathcal{N}(0,1/4),
\] 
such that the historically desired demands are given by realizations of 
\begin{equation}\label{eq:creating_historically_desired_demand}
	D^{(i)}_{\textup{d}} = \sum_{\ell = 1}^4 \beta_\ell\, \hat{d}^{(i)} \left(D_\ell + Z^{(i)}_\ell\right),
	\quad
	i\in\{7,\ldots,11\}.
\end{equation}
The historically observed pairs $(z_\textup{o},u_\textup{o})$ are computed as solutions of 
the associated problem \eqref{eq:forward_transportation_problem}.

\subsection{\changed{Documentation of experiments}}
\label{sec:documentation}

In the following, we investigate different variants of the bilevel optimization problem \eqref{eq:upper_level_problem}. 
The standard version is presented in \cref{sec: simulations1},
and settings with additional perturbations in the historical observations are shown in \cref{sec: simulations2,sec: simulations4}.
A time-restricted observation operator $\mathcal{C}$ is investigated in \cref{sec: simulations7}. All subsections are constructed in a similar way. 
First, we present exemplary historical demand observations, 
then we provide a comparison of the in- and outflows for the means of the historical observations 
and the initially chosen $\beta$ as well as for the reconstructed $\beta$ 
\changed{in a framework without an inflow constraint, which can also be considered as a framework with a high constraint that does not really affect the inflow.}
These illustrations are presented for the inflow vertex and the demand vertex $v_7$
(the behavior at all other demand vertices is similar). 
We can verify that, on the one hand, the optimal inflows are calculated correctly and,
on the other hand, see whether the reconstruction of the weights $\beta$ was successful. 
The second aspect is further underlined by a table presenting the means and variances for $\beta$ of a Monte Carlo simulation 
of $N=40$ runs for different numbers of historical observations $p$. 
\changed{%
	Second, we repeat the investigations of each subcase 
	based on a medium inflow constraint 
	$u_\textup{b}\equiv 2$ and a low inflow constraint $u_\textup{b} \equiv 1.5$, where we also ensure nonnegative inflows, 
	i.e., $u_\textup{a} \equiv 0$,
	the latter being nonrestrictive as the desired demand at the vertices in $V_D$
	is nonnegative.
	}%

\subsubsection{Standard model without additional adjustments}
\label{sec: simulations1}

In this scenario, no further perturbations or model changes are included,
and we consider the framework presented in the previous sections.
Three examples for historical observations are given in \cref{fig: c1_histData} 
which show the sinusoidal behavior of demand, as well as the drop for $t>120$ during the weekend. 
Furthermore, we detect the stochastic noise in the demands, 
however, still verify that the demands show a very similar structure. 
The comparison of the inflow and outflow for demand vertex $v_7$ are presented in \cref{fig: c1_inflow_outflow}, 
where the blue curve shows the mean values of the $p=6$ historical observations, 
the yellow dotted line represents the curve for the \textit{true} $\beta$ given in \eqref{initialBeta}, 
and the red line the in- or outflow for the reconstructed $\beta$. All considerations were made without constraining the inflow control.
It can be concluded that all three curves match very well, 
which means that, on the one hand, the inflow is calculated appropriately and, on the other hand, 
also the weights of the base demands are reobtained very well.
The outflow behavior at the demand vertices $v_8,\dots,v_{11}$ shows similar patterns and is (for brevity of presentation) not illustrated.  
\changed{%
	At the beginning and the end of the considered time horizon, 
	some curves in \cref{fig: c1_inflow_outflow} decay to zero or show a jump. 
	This can be explained by the fact that around time $t=0$, 
	it takes some time until (starting from an empty system) 
	the first inserted quantity reaches the demand vertex. 
	Therefore, the outflows are zero in the very beginning of the time period. 
	Conversely, the inflow for times close to $T=168$ vanishes,
	since these quantities do not reach the demand nodes within the considered time horizon. 
	The increase at $T=168$ in the outflow figure can be explained 
	by considering $T=168$ to be Monday already, where the demand is larger again.
	Similar artifacts show up in some other figures in this section due to analogous reasons.
}%

\cref{tab: c1_beta} shows the means and variances of the reconstructed weights 
for the base demands for different numbers of perturbed historical observations in a Monte Carlo simulation of $N=40$ runs 
and underlines the results from \cref{fig: c1_inflow_outflow} quantitatively. 
As it can be expected for larger numbers of historical observations, 
the means approach the values in \eqref{initialBeta} and the variances in the runs decrease in the number of historical observations $p$.

\begin{figure}[ht!]
	\centering
	\includegraphics[scale=0.5]{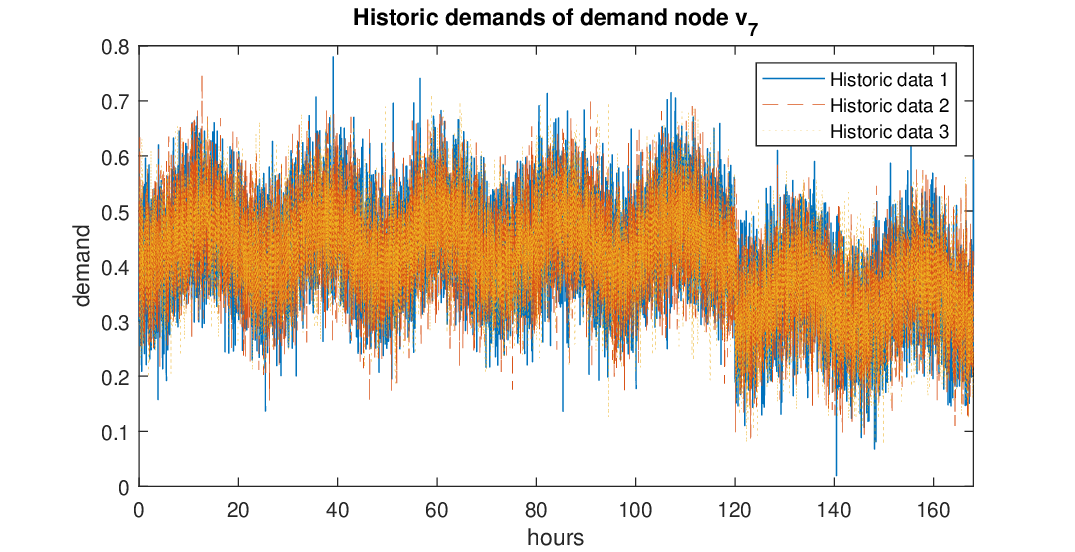}
	\caption{%
		Three of the perturbed historically observed demands for demand vertex $v_7$ in the case of \cref{sec: simulations1}.
		}%
	\label{fig: c1_histData}
\end{figure}

\begin{figure}[ht!]
	\centering
	\includegraphics[scale=0.72]{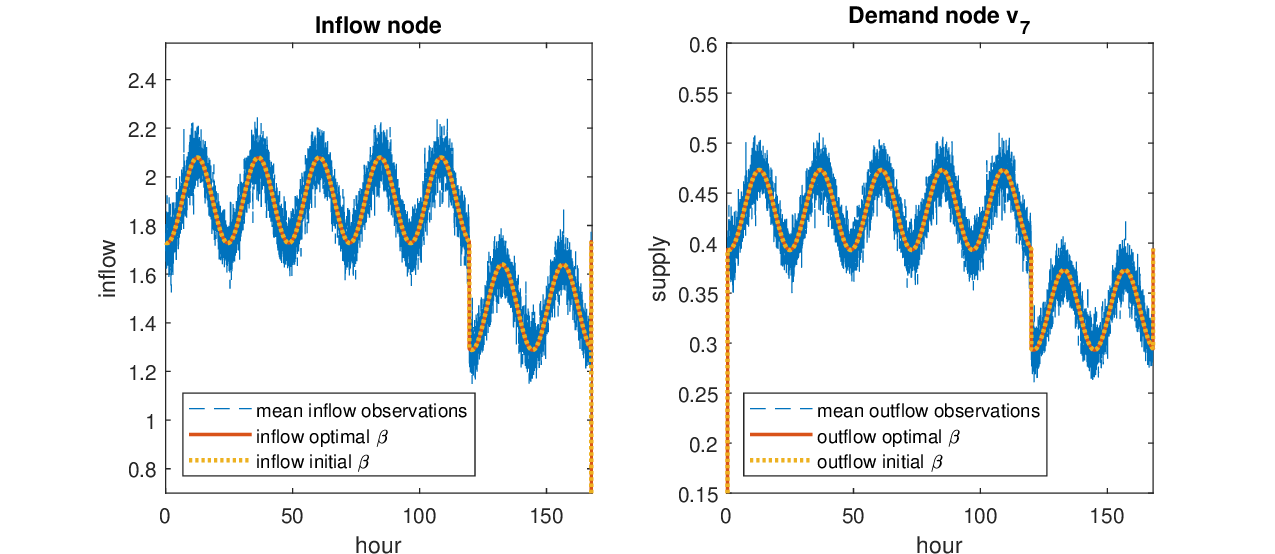}
	\caption{%
		A comparison between the mean realization of the $p=6$ historical in- and outflows with the in- and outflow 
		for the reconstructed $\beta$ in the case of \cref{sec: simulations1} and the initial $\beta$.
		}%
	\label{fig: c1_inflow_outflow}
\end{figure}

\begin{table}[ht!]
	\centering\small{
		\begin{tabular}{r||c|c|c|c||c|c|c|c}
			&\multicolumn{4}{c||}{mean} & \multicolumn{4}{c}{variance}\\
			\hline\hline
		&$p=1$& $p=6$ & $p=20$ & $p=200$&$p=1$& $p=6$ & $p=20$ & $p=200$\\
		\hline\hline
		$\beta_1$  &0.2003&0.2000&0.1999&0.2000&2.16e-06&0.25e-06&0.16e-06&0.10e-07\\
		\hline
		$\beta_2$  &0.1496&0.1499&0.1500&0.1500&6.65e-06&0.42e-06&0.29e-06&0.21e-07\\
		\hline
		$\beta_3$  &0.2001&0.2001&0.2002&0.2000&2.82e-06&0.87e-06&0.31e-06&0.18e-07\\
		\hline
		$\beta_4$  &0.4503&0.4500&0.4499&0.4500&4.16e-06&0.57e-06&0.34e-06&0.21e-07\\
	\end{tabular}}
	\caption{%
		Means and variances of the reobtained weights for the base demands for different choices of the number of perturbed historical observations $p$ 
		in the setting of \cref{sec: simulations1}.
		}%
\label{tab: c1_beta}
\end{table}

\changed{%
Accounting for a potential constraint on the inflow, 
we compare a scenario where the inflow is limited to 2 (medium constraint) 
and 1.5 (low constraint). 
We repeat the idea of \cref{fig: c1_inflow_outflow} in \cref{fig: c1_inflow_outflow_constraint}
emphasizing that, except for the constraint, all other quantities remain unchanged. 
However, the demand illustration seems to be less fluctuating which can be explained 
by the coarser discretization grid that is used for the constrained optimization. 
In the medium constraint case, we observe that the in- and outflow follow the unconstrained case
but are truncated at the very highest peaks and otherwise follow the averaged demand well.
Regarding the reconstruction of the weights of the base demand levels when zooming in, 
one can still observe a quite good match 
in the in- and outflows of the optimized and initial choices of $\beta$.
\cref{tab: c1_beta_constraint} underlines this observation, 
but shows a small deviation especially in the parameters $\beta_1$ and $\beta_2$ 
compared to the unrestricted case. 
For the low constraint, the inflow is cut from Monday to Friday 
and in some peak times also during the weekend, 
so that most of the time demand cannot be satisfied on average. 
Then the reconstruction task is also not successful, 
and we can observe a visible mismatch in the green circles 
(associated to the optimal outflow for the initial $\beta$) 
and purple diamonds (representing the outflow for the reconstructed $\beta$) 
during the weekend. 
Referring again to \cref{tab: c1_beta_constraint}, 
one can see that there is a large deviation in the reconstructed values of $\beta$, 
where the very low values of $\beta_2$ and $\beta_3$ are particularly striking. 
This effect can be explained by the fact that $D_2$ and $D_3$ are the sinusoidal components 
of demand and that the observations are smoothed and truncated at the majority of time.
}%

\begin{figure}[ht!]
	\centering
	\includegraphics[scale=0.5]{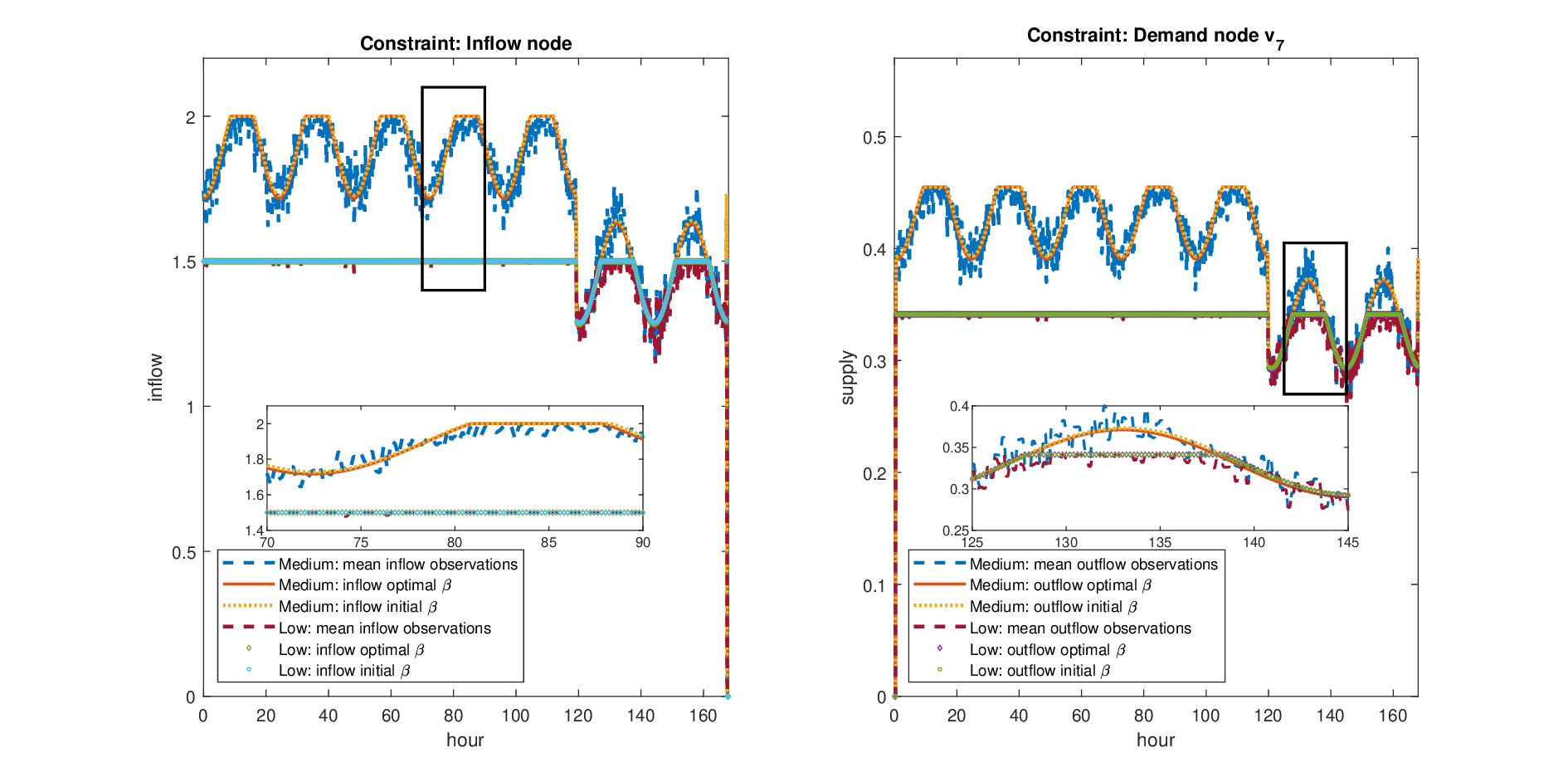}
	\caption{%
		\changed{A comparison between the mean realization of the $p=6$ historical in- and outflows with the in- and outflow 
		for the reconstructed $\beta$ in the case of \cref{sec: simulations1} and the initial $\beta$ with two different inflow constraints.}
	}%
	\label{fig: c1_inflow_outflow_constraint}
\end{figure}

\begin{table}[ht!]
	\centering\small{
		\begin{tabular}{r||c|c||c|c||c|c||c|c}
			&\multicolumn{4}{c||}{mean} & \multicolumn{4}{c}{variance}\\
			\hline\hline
			&\multicolumn{2}{c||}{medium} & \multicolumn{2}{c||}{low}&\multicolumn{2}{c||}{medium} & \multicolumn{2}{c}{low}\\
			\hline
			&$p=6$& $p=200$ & $p=6$ & $p=200$&$p=6$& $p=200$ & $p=6$ & $p=200$\\
			\hline\hline
			$\beta_1$  &0.1960&0.1972&0.2460&0.2474&2.87e-06&0.11e-06&0.83e-04&0.11e-05\\
			\hline
			$\beta_2$  &0.1549&0.1536&0.0997&0.0988&5.93e-06&0.22e-06&0.91e-04&0.10e-05\\
			\hline
			$\beta_3$  &0.1997&0.1997&0.1297&0.1284&6.77e-06&0.35e-06&1.10e-04&0.21e-05\\
			\hline
			$\beta_4$  &0.4493&0.4495&0.5245&0.5254&6.24e-06&0.30e-06&1.11e-04&0.15e-05\\
	\end{tabular}}
	\caption{%
		\changed{Means and variances of the reobtained weights for the base demands for different choices of the number of perturbed historical observations $p$ 
		in the setting of \cref{sec: simulations1} with additional inflow constraints}.
	}%
	\label{tab: c1_beta_constraint}
\end{table}

\subsubsection{Results with additional noise in the weights $\beta$}
\label{sec: simulations2}
In addition to the investigation of \cref{sec: simulations1}, 
we introduce a structural and uncertain deviation in the choice of $\beta$, when generating the historically desired 
demand in \eqref{eq:creating_historically_desired_demand}.
We assume that the uncertainty mainly comes into play for $\beta_4$ such that for any historical observation, 
the weights for the demand levels are chosen as
\begin{align}
	\label{betaperturbed}
	(\beta_1, \beta_2, \beta_3,\beta_4) 
	= 
	\left(\frac{0.2}{1+\tilde{Z}},\frac{0.15}{1+\tilde{Z}},\frac{0.2}{1+\tilde{Z}},\frac{0.45+\tilde{Z}}{1+\tilde{Z}}\right)
\end{align}
for a random variable $\tilde{Z} \sim \mathcal{U}([-0.05,0.05])$.

The results for some historical observations are presented below in \cref{fig: c2_histData}. 
There is not only noise in the demands but also structurally different behavior 
due to different realizations of $\tilde{Z}$ in the weights of the demands. 
Therefore, the yellow curve of historic data 3 seems to be lower 
(corresponding to a larger value of $\tilde{Z}$) than the blue curve (corresponding to a smaller value of $\tilde{Z}$).
\cref{fig: c2_inflow_outflow} shows the different in- and outflows which are supplemented by \cref{tab: c2_beta} 
showing the means and the variances of a Monte Carlo simulation for the reconstructed weights of the base demands 
for different numbers of perturbed historical observations. 
We observe that in \cref{fig: c2_inflow_outflow}, the expected outflow and inflow match quite well, 
but considering \cref{tab: c2_beta}, it can be seen that the reconstruction is more difficult than in the standard setting.
For small $p$, the reconstructed $\beta$ deviates more significantly from the initial choice. 
For a larger number of observations $p$, the data indicates that the performances are improved and lead to good reconstructed values of $\beta$. 

\begin{figure}[ht!]
	\centering
	\includegraphics[scale=0.5]{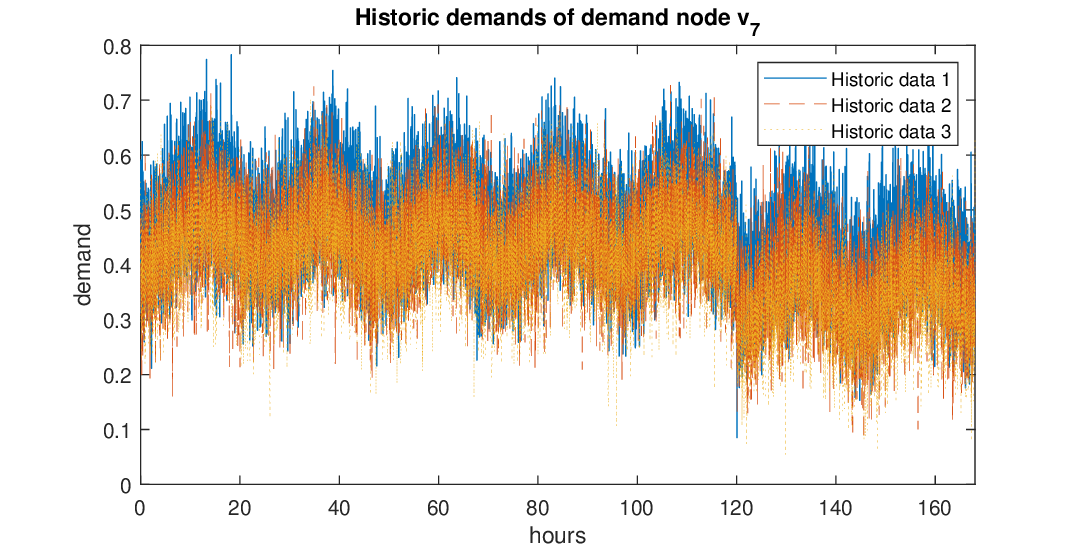}
	\caption{%
		Three of the perturbed historically observed demands for demand vertex $v_7$ in the case of \cref{sec: simulations2}.
		}%
	\label{fig: c2_histData}
\end{figure}

\begin{figure}[ht!]
	\centering
	\includegraphics[scale=0.72]{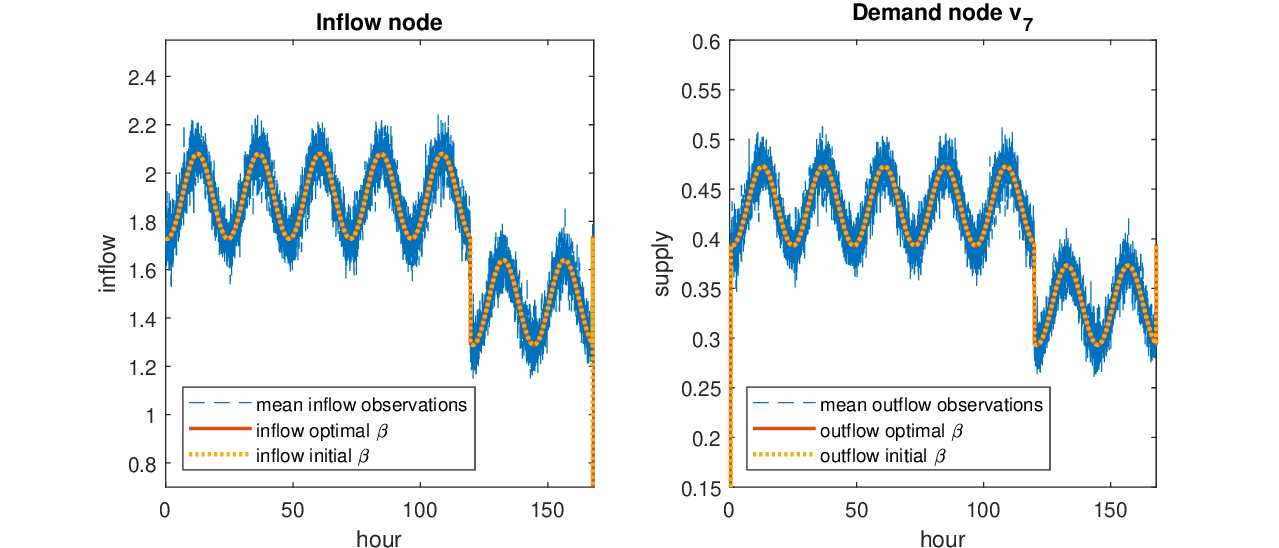}
	\caption{%
		A comparison between the mean realization of the $p=6$ historical in- and outflows 
		with the in- and outflow for the reconstructed $\beta$ in the case of \cref{sec: simulations2} and the initial $\beta$.
		}%
	\label{fig: c2_inflow_outflow}
\end{figure}

\begin{table}[ht!]
	\centering\small{
		\begin{tabular}{r||c|c|c|c||c|c|c|c}
			&\multicolumn{4}{c||}{mean} & \multicolumn{4}{c}{variance}\\
			\hline\hline
			&$p=1$& $p=6$ & $p=20$ & $p=200$&$p=1$& $p=6$ & $p=20$ & $p=200$\\
			\hline\hline
			$\beta_1$  &0.2012&0.2010&0.2001&0.2001&0.36e-04&0.08e-04&0.19e-05&0.01e-05\\
			\hline
			$\beta_2$  &0.1497&0.1509&0.1500&0.1501&0.26e-04&0.04e-04&0.14e-05&0.01e-05\\
			\hline
			$\beta_3$  &0.2013&0.2013&0.2003&0.2001&0.40e-04&0.06e-04&0.19e-05&0.01e-05\\
			\hline
			$\beta_4$  &0.4477&0.4468&0.4496&0.4497&2.76e-04&0.49e-04&1.37e-05&0.08e-05\\
			\end{tabular}}
			\caption{%
				Means and variances of the reobtained weights for the base demands 
				for different choices of the number of perturbed historical observations $p$ in the setting of \cref{sec: simulations2}.
				}%
			\label{tab: c2_beta}
		\end{table}
		
\changed{%
	Also in this scenario, we investigate a constraint on the inflow control 
	on a medium level of 2 and a low constraint of 1.5. 
	Similar to \cref{sec: simulations1}, the reconstruction works at least satisfactorily 
	in the medium constraint case, 
	whereas it fails in the low constraint case. 
	Nevertheless, in both cases, the average outflow matches the optimal outflow 
	for the reconstructed $\beta$, see \cref{fig: c2_inflow_outflow_constraint}.
	\cref{tab: c2_beta_constraint} shows for $p\in\{6,200\}$ the mean and the variance 
	as the adapted version of \cref{tab: c2_beta} with medium and low inflow constraint, 
	where the variances are similar but slightly higher than in the unconstrained framework.
	The observed effects are comparable to those obtained for the constrained but unperturbed regime in \cref{tab: c1_beta_constraint}.
	 }%
		\begin{figure}[ht!]
			\centering
			\includegraphics[scale=0.5]{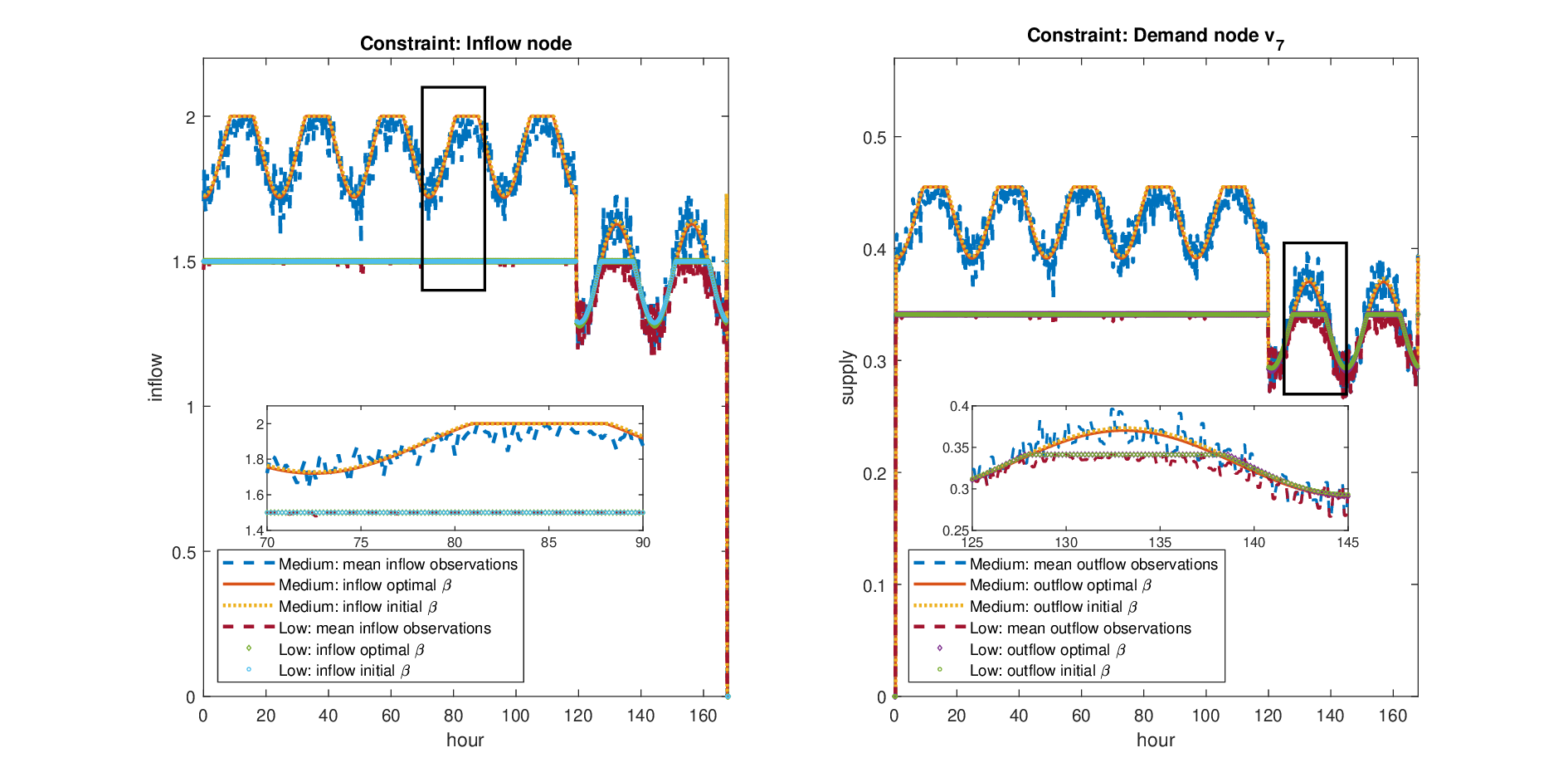}
			\caption{%
				\changed{A comparison between the mean realization of the $p=6$ historical in- and outflows 
					with the in- and outflow for the reconstructed $\beta$ in the case of \cref{sec: simulations2} and the initial $\beta$ with two different inflow constraints.}
			}%
			\label{fig: c2_inflow_outflow_constraint}
		\end{figure}

		\begin{table}[ht!]
			\centering\small{
				\begin{tabular}{r||c|c||c|c||c|c||c|c}
					&\multicolumn{4}{c||}{mean} & \multicolumn{4}{c}{variance}\\
					\hline\hline
					&\multicolumn{2}{c||}{medium} & \multicolumn{2}{c||}{low}&\multicolumn{2}{c||}{medium} & \multicolumn{2}{c}{low}\\
					\hline
					&$p=6$& $p=200$ & $p=6$ & $p=200$&$p=6$& $p=200$ & $p=6$ & $p=200$\\
					\hline\hline
					$\beta_1$  &0.1958&0.1965&0.2494&0.2518&0.06e-04&0.02e-05&1.01e-04&0.18e-05\\
					\hline
					$\beta_2$  &0.1548&0.1539&0.0923&0.0909&0.10e-04&0.06e-05&0.85e-04&0.10e-05\\
					\hline
					$\beta_3$  &0.1996&0.1997&0.1285&0.1262&0.10e-04&0.09e-05&1.21e-04&0.28e-05\\
					\hline
					$\beta_4$  &0.4498&0.4499&0.5297&0.5310&0.36e-04&0.27e-05&1.09e-04&0.24e-05\\
			\end{tabular}}
			\caption{%
				\changed{Means and variances of the reobtained weights for the base demands 
					for different choices of the number of perturbed historical observations $p$ in the setting of \cref{sec: simulations2} with additional inflow constraints}.
			}%
			\label{tab: c2_beta_constraint}
		\end{table}

\subsubsection{Results with changed base demand level $D_4$}
\label{sec: simulations4}

This section is based on the investigations in \cref{sec: simulations1}. 
Instead of perturbing $\beta$, we assume that there is a structural deviation in the base demand levels. 
Particularly, we assume that in the generation of the observations, 
we adjust the base demand $D_4$ to $D_4(t) = \frac{3}{2}\mathbbm{1}_{[0,120]}(t)$, which means that there is larger share of demand on weekdays.
Furthermore, we omit the normalization restriction to the weights, i.e., we merely assume $\beta_\ell\geq 0$, $\ell\in\{1,\ldots,4\}$,
and drop the constraint $\sum_{\ell = 1}^4 \beta_\ell = 1$, since the increase in the base demand level should now be captured by a larger weight on $\beta_4$.
Note that we still use $D_4(t)=\mathbbm{1}_{[0,120]}(t)$ in the lower-level objective function from \eqref{eq:lower_level_objective}
for the reconstruction task.
Similar to \cref{sec: simulations1}, examples of the historical observations are presented in \cref{fig: c4_histData} 
which now show a larger difference between the weekday demand and the weekend demand induced by the larger value in $D_4$. 
Again, \cref{fig: c4_inflow_outflow} shows the in- and outflows for the means of the observations (blue), 
the initial $\beta$ from \eqref{initialBeta} (yellow dotted), and the reconstructed $\beta$ (red). 
One can see that the in- and outflow of the initial $\beta$ do not match from Monday to Friday, 
but do on the weekend, since they are not able to take into account the change in $D_4$. 
The reconstructed weights yield the correct inflows and outflows with respect to the observations 
and manage to compensate the structural deviation in $D_4$. 
\changed{%
	In \cref{tab: c4_beta}, we observe that the values for $\beta_1$, $\beta_2$, and $\beta_3$ are very well reobtained 
	with similar variances as in \cref{tab: c1_beta}. 
	The value of $\beta_4$ now exceeds significantly the initial value of $0.45$. 
	Recalling that the base demand level $D_4$ was increased from 1 to 1.5 at the lower-level stage, 
	one notices that also the reconstructed value of $\beta_4$ increased by factor 1.5 to compensate 
	to unchanged level of $D_4$ in the parameter reconstruction.
	}%
This allows for a good reconstruction of the observation means, 
but as expected violates the normalization of the base demand weights.

\begin{figure}[ht!]
	\centering
	\includegraphics[scale=0.5]{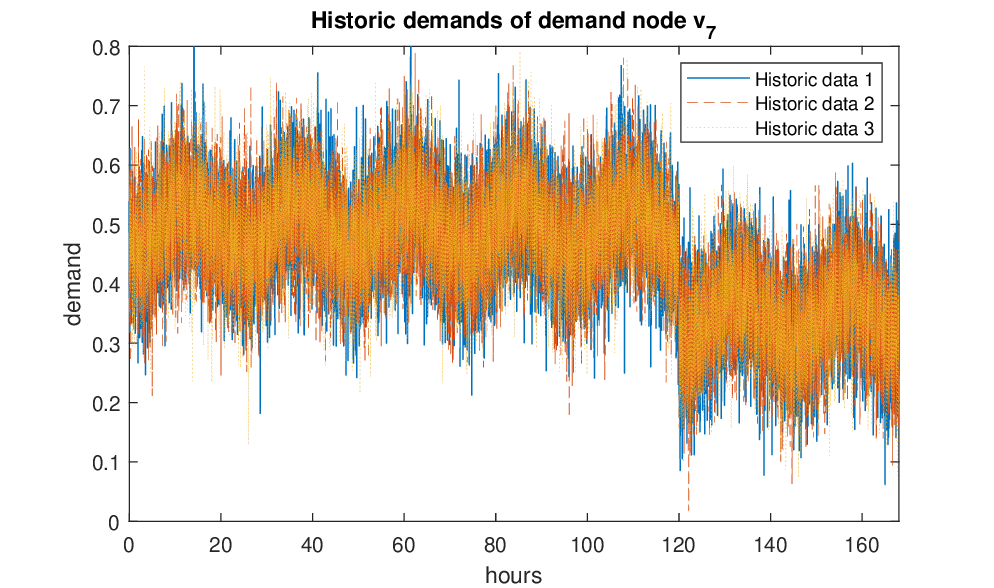}
	\caption{%
		Three of the perturbed historically observed demands for demand vertex $v_7$ in the case of \cref{sec: simulations4}.
		}%
	\label{fig: c4_histData}
\end{figure}

\begin{figure}[ht!]
	\centering
	\includegraphics[scale=0.72]{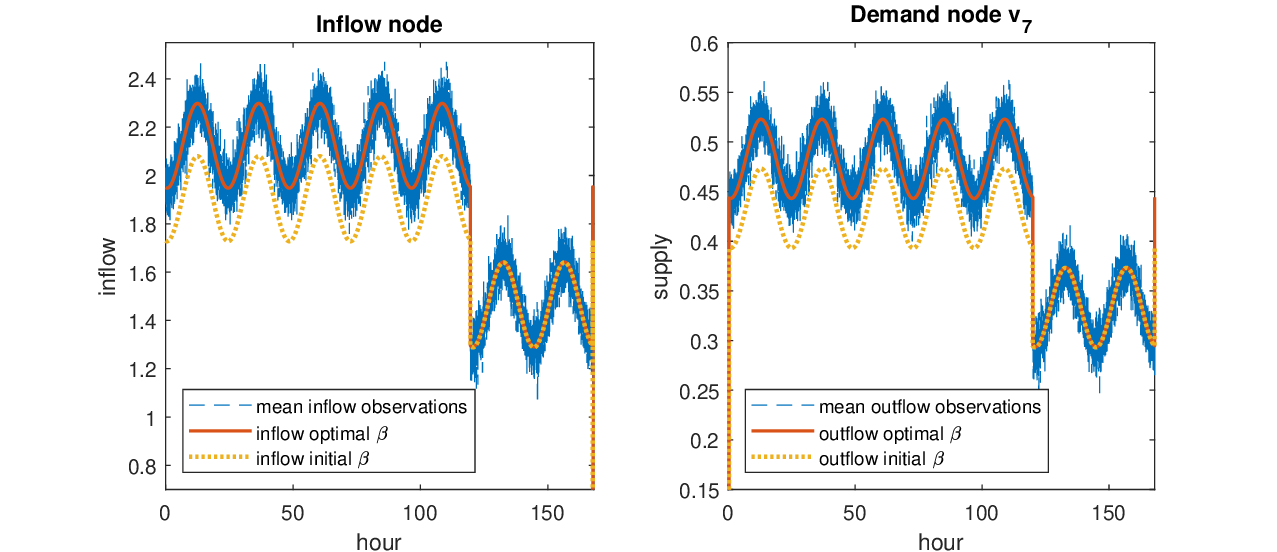}
	\caption{%
		A comparison between the mean realization of the $p=6$ historical in- and outflows 
		with the in- and outflow for the reconstructed $\beta$ in the case of \cref{sec: simulations4} 
		and the initial $\beta$ including the structural deviation in $D_4$.
		}%
	\label{fig: c4_inflow_outflow}
\end{figure}

\begin{table}[ht!]
	\centering\small{
		\begin{tabular}{r||c|c|c|c||c|c|c|c}
			&\multicolumn{4}{c||}{mean} & \multicolumn{4}{c}{variance}\\
			\hline\hline
			&$p=1$& $p=6$ & $p=20$ & $p=200$&$p=1$& $p=6$ & $p=20$ & $p=200$\\
			\hline\hline
			$\beta_1$  &0.1999&0.1998&0.1998&0.2000&0.59e-05&0.14e-05&0.03e-05&0.03e-06\\
			\hline
			$\beta_2$  &0.1505&0.1499&0.1502&0.1499&0.77e-05&0.16e-05&0.03e-05&0.03e-06\\
			\hline
			$\beta_3$  &0.2003&0.2002&0.2003&0.2000&0.53e-05&0.13e-05&0.03e-05&0.03e-06\\
			\hline
			$\beta_4$ &0.6734&0.6752&0.6750&0.6749&1.20e-05&0.20e-05&0.04e-05&0.06e-06\\
		\end{tabular}}
	\caption{%
		Means and variances of the reobtained weights for the base demands for different choices 
		of the number of perturbed historical observations $p$ with additive deviation in $D_4$ in the setting of \cref{sec: simulations4}.
		}%
\label{tab: c4_beta}
\end{table}

\changed{%
When adapting the base demand level accounting for the weekend effect 
with a medium constraint $u_b\equiv 2$, similar to the unconstrained case, 
we observe in \cref{fig: c4_inflow_outflow_constraint} that the in- and outflow rates 
for the reconstructed $\beta$ exceed those for the initial $\beta$, 
but are cut at the maximum inflow level of 2. 
Therefore, in \cref{tab: c4_beta_constraint}, the values of the reconstructed $\beta$ 
are very similar to the ones obtained in \cref{tab: c4_beta} without constraint, 
but have a lower value for $\beta_4$ due to the inflow constraint and the resulting 
lower observed supply. For the more restricted inflow control with $u_b\equiv 1.5$, 
the reconstructed $\beta$ is far from the initial one, 
since by cutting off the inflow, we lose information about the true demand.
}%

\begin{figure}[ht!]
	\centering
	\includegraphics[scale=0.5]{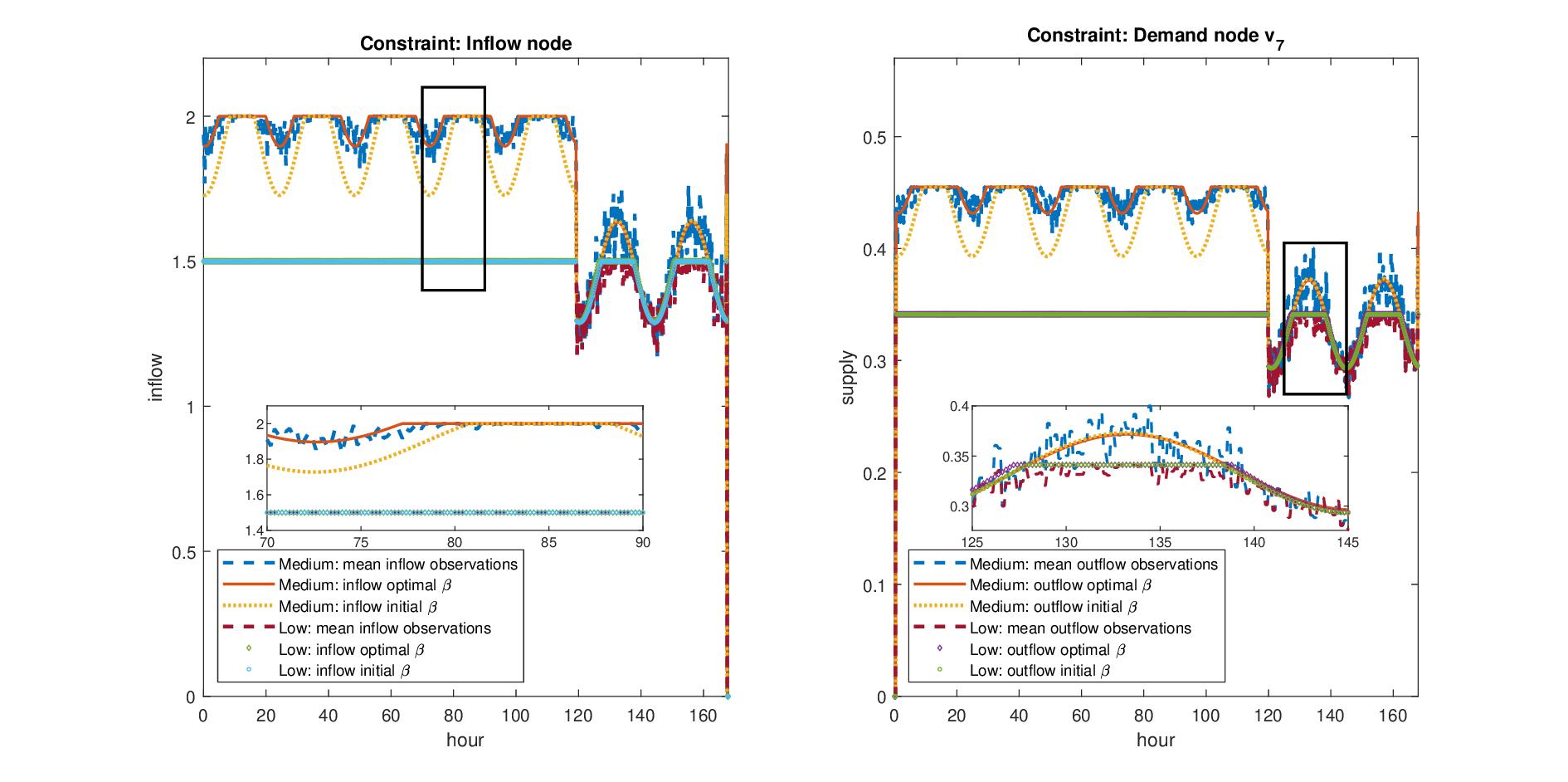}
	\caption{%
		\changed{A comparison between the mean realization of the $p=6$ historical in- and outflows 
			with the in- and outflow for the reconstructed $\beta$ in the case of \cref{sec: simulations4} 
			and the initial $\beta$ including the structural deviation in $D_4$ with two different inflow constraints.}
	}%
	\label{fig: c4_inflow_outflow_constraint}
\end{figure}

\begin{table}[ht!]
	\centering\small{
		\begin{tabular}{r||c|c||c|c||c|c||c|c}
			&\multicolumn{4}{c||}{mean} & \multicolumn{4}{c}{variance}\\
			\hline\hline
			&\multicolumn{2}{c||}{medium} & \multicolumn{2}{c||}{low}&\multicolumn{2}{c||}{medium} & \multicolumn{2}{c}{low}\\
			\hline
			&$p=6$& $p=200$ & $p=6$ & $p=200$&$p=6$& $p=200$ & $p=6$ & $p=200$\\
			\hline\hline
			$\beta_1$  &0.2056&0.2049&0.2429&0.2445&0.18e-04&0.07e-05&0.84e-04&0.14e-05\\
			\hline
			$\beta_2$  &0.1474&0.1489&0.0993&0.0980&0.27e-04&0.08e-05&0.93e-04&0.12e-05\\
			\hline
			$\beta_3$  &0.1915&0.1918&0.1379&0.1358&0.28e-04&0.10e-05&1.15e-04&0.26e-05\\
			\hline
			$\beta_4$  &0.6111&0.6130&0.9867&0.9991&0.38e-04&0.08e-05&3.85e-04&8.18e-05\\
	\end{tabular}}
	\caption{%
		\changed{Means and variances of the reobtained weights for the base demands for different choices 
			of the number of perturbed historical observations $p$ with additive deviation in $D_4$ in the setting of \cref{sec: simulations4} with additional inflow constraints}.
	}%
	\label{tab: c4_beta_constraint}
\end{table}

\subsubsection{Results with no noise in $\beta$ and observations only on Sunday}
\label{sec: simulations7}

In this section, we consider another variant of the framework in \cref{sec: simulations1} 
and do not introduce additional noise or deviations.
Instead, we restrict the observation time of the historical data by adjusting the observation operators $\mathcal{C}$ and $\mathcal D$. 
Therefore, the observation operators only consider $t \in [144,168]$,
i.e., historical data is only taken into account on Sundays, and the model is used to reestimate the choices of $\beta$ from that knowledge for the entire week. 
Furthermore, to avoid hidden information on $\beta_4$ via the normalization constraint, 
similar to \cref{sec: simulations4}, we drop the condition $\sum_{\ell=1}^4 \beta_\ell =1$ in the optimization problem \eqref{eq:upper_level_problem}
and merely require $\beta_\ell\geq 0$, $\ell\in\{1,\ldots,4\}$.
For the illustration of historical data, we refer to \cref{fig: c1_histData} in \cref{sec: simulations1} 
as there are no differences in the historical data.
In \cref{fig: c7_inflow_outflow}, for the in- and outflows, we observe an undersupply 
in the curve of the optimal $\beta$ (red) from Monday to Friday. 
For the weekend, the reconstruction works fine, since past information is available for this time period. 
\cref{tab: c7_beta} underlines that there is no information on the choice of $\beta_4$. 
While the means seem to be close together in the Monte Carlo simulation for the different numbers of observations, 
the very large variances reveal that $\beta_4$ is arbitrary to choose in the case of restricted information.
This, however, is not surprising as, on the one hand, the normalization condition on the weights has been dropped
and, on the other hand, $D_4$ vanishes during the weekend.

\begin{figure}[ht!]
	\centering
	\includegraphics[scale=0.72]{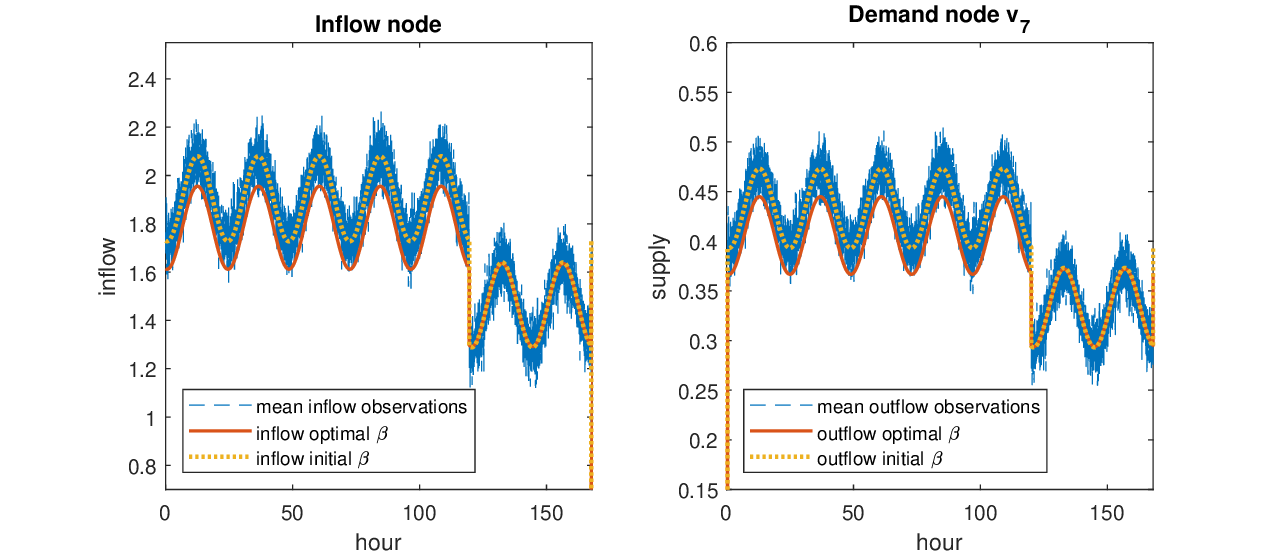}
	\caption{%
		A comparison between the mean realization of the $p=6$ historical in- and outflows 
		with the in- and outflow for the reconstructed $\beta$ in the case of \cref{sec: simulations7} 
		and the initial $\beta$ where observations can only be made on Sundays ($t \in [144,168]$).
		}%
	\label{fig: c7_inflow_outflow}
\end{figure}

\begin{table}[ht!]
	\centering\small{
		\begin{tabular}{r||c|c|c|c||c|c|c|c}
			&\multicolumn{4}{c||}{mean} & \multicolumn{4}{c}{variance}\\
			\hline\hline
			&$p=1$& $p=6$ & $p=20$ & $p=200$&$p=1$& $p=6$ & $p=20$ & $p=200$\\
			\hline\hline
			$\beta_1$ &0.2007&0.1998&0.1997&0.2001 &3.35e-06&0.60e-06&0.20e-06&0.15e-07\\
			\hline
			$\beta_2$ &0.1500&0.1504&0.1503&0.1499&3.76e-06&0.78e-06&0.23e-06&0.20e-07\\
			\hline
			$\beta_3$  &0.1988&0.1999&0.2003&0.2000&4.48e-06&0.77e-06&0.26e-06&0.24e-07\\
			\hline
			$\beta_4$  &0.3622&0.3814&0.3663&0.3718&1.788&0.191&0.075&0.007\\
			\end{tabular}}
			\caption{%
				Means and variances of the reobtained weights for the base demands for different choices 
				of the number of perturbed historical observations $p$ with observations only for $t \in [144,168]$ from \cref{sec: simulations7}.
				}%
			\label{tab: c7_beta}
		\end{table}
		
\changed{%
Under additional inflow constraints, \cref{fig: c7_inflow_outflow_constraint} shows 
the behavior of the in- and outflows according to the choices 
for the initial and reconstructed $\beta$. 
For the medium inflow constraint of level 2, we observe that from Monday to Friday 
the reconstructed in- and outflow deviate significantly 
from the outflow observations, since all the measurements were taken on Sunday. 
This is also underlined by \cref{tab: c7_beta_constraint} where, in this case, 
we find a highly varying reconstructed $\beta_4$, depicting that the choice of $\beta_4$ 
cannot be controlled and has no impact on the objective function. 
During the weekend, the observed outflow and the outflow 
from the optimally reconstructed $\beta$ are in good agreement. 
This is no longer true when considering the lower constraint 
where the inflow is also cut on Sundays. 
Therefore, as in all previous examples, the reconstruction of $\beta$ 
is not successful for the low inflow constraint.
}%

		\begin{figure}[ht!]
			\centering
			\includegraphics[scale=0.5]{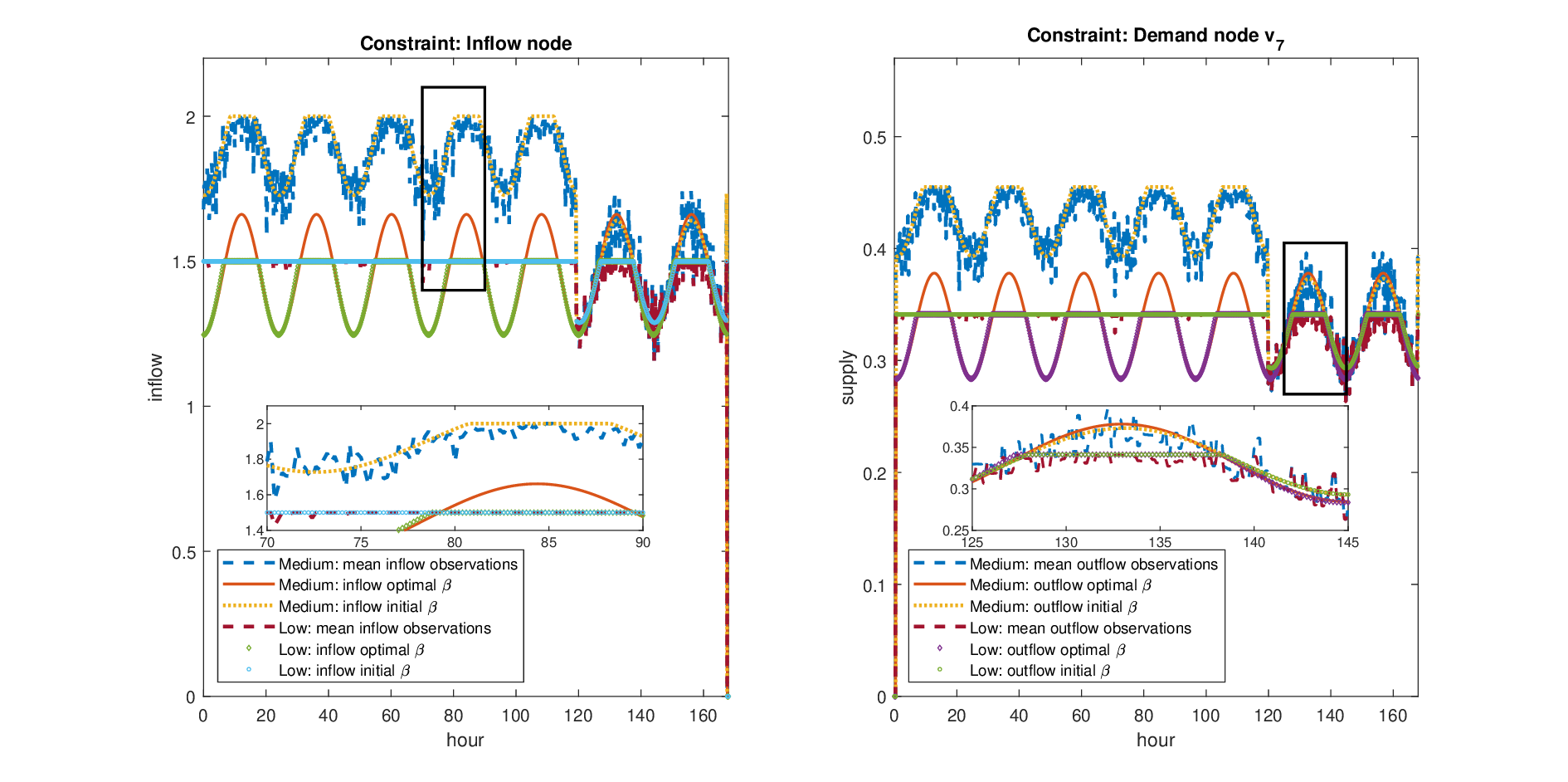}
			\caption{%
				\changed{A comparison between the mean realization of the $p=6$ historical in- and outflows 
				with the in- and outflow for the reconstructed $\beta$ in the case of \cref{sec: simulations7} 
				and the initial $\beta$ where observations can only be made on Sundays ($t \in [144,168]$) with two different inflow constraints.}
			}%
			\label{fig: c7_inflow_outflow_constraint}
		\end{figure}
		
		\begin{table}[ht!]
			\centering\small{
				\begin{tabular}{r||c|c||c|c||c|c||c|c}
					&\multicolumn{4}{c||}{mean} & \multicolumn{4}{c}{variance}\\
					\hline\hline
					&\multicolumn{2}{c||}{medium} & \multicolumn{2}{c||}{low}&\multicolumn{2}{c||}{medium} & \multicolumn{2}{c}{low}\\
					\hline
					&$p=6$& $p=200$ & $p=6$ & $p=200$&$p=6$& $p=200$ & $p=6$ & $p=200$\\
					\hline\hline
					$\beta_1$  &0.1763&0.1760&0.2192&0.2443&0.56e-04&0.16e-05&0.465e-04&0.343e-06\\
					\hline
					$\beta_2$  &0.1665&0.1684&0.1206&0.0966&0.78e-04&0.31e-05&0.481e-04&0.390e-06\\
					\hline
					$\beta_3$  &0.2220&0.2240&0.1642&0.1377&0.86e-04&0.22e-05&0.574e-04&0.571e-06\\
					\hline
					$\beta_4$  &0.3452&0.3578&0.3001&0.3133&0.1934&0.009&0.215&0.013\\
			\end{tabular}}
			\caption{%
				\changed{Means and variances of the reobtained weights for the base demands for different choices 
					of the number of perturbed historical observations $p$ with observations only for $t \in [144,168]$ from \cref{sec: simulations7} with additional inflow constraints}.
			}%
			\label{tab: c7_beta_constraint}
		\end{table}

\section{Conclusions}\label{sec:conclusions}

The present paper is devoted to the inverse demand tracking over transportation networks of tree shape
governed by linear hyperbolic conservation laws. By considering this problem in the
setting of Lebesgue spaces for control variables, we were in position to ensure
the existence of optimal solutions. These were numerically computed by solving an associated
finite-dimensional nonsmooth optimization problem which results from the original hierarchical
model 
\changed{%
	by plugging the pointwise computable Lipschitzian solution operator 
	of the subordinate convex network flow problem
	into the upper-level objective function
}%
before discretizing the resulting single-level problem 
by means of a finite differences scheme. Numerical experiments demonstrated that 
for full time observation operators \changed{in the absence of control constraints}, 
there is a very good agreement
of solutions for the optimal and the benchmark problem.
Restricted-in-time observation operators 
\changed{or the presence of control constraints}
may lead to different solutions, depending on the choice of the underlying base demands
\changed{or the restrictiveness of the employed constraint.}

Our findings give rise to at least \changed{two} interesting directions for future
research. First, it is well known that several network dynamics obey nonlinear hyperbolic
partial differential equations like the flow of gases, see e.g.\ \cite{Bressan2000}, while, 
in this paper, we
focused on merely linear dynamics. We note that incorporating nonlinear dynamics 
in an inverse optimal control problem 
makes the lower-level problem nonconvex and, thus,
\changed{%
an explicit computation of the lower-level solution operator is highly challenging,
see \cref{rem:box_constraints_at_lower_level} as well.
Furthermore, one cannot simply replace the lower-level problem by 
(necessary but not necessarily sufficient) 
optimality conditions without enlarging the feasible set significantly.
Second, in the setting discussed in this paper, it might be worth trying to
solve the finite-dimensional nonsmooth single-level problem with a more enhanced
numerical scheme than just MATLAB's \texttt{patternsearch} method.
Indeed, the special structure of the Lipschitzian lower-level solution operator 
discussed in \cref{sec:special_QPs} allows for an explicit computation of its
so-called Clarke generalized Jacobian, see \cite{Clarke1983},
and the latter can be used to construct a bundle-type algorithm, see \cite{SchrammZowe1992}, 
for the numerical solution of the nonsmooth optimization problem of interest since
the latter merely possesses affine constraints,
see \cite{DempeBard1992} for a related idea.
}%

\paragraph{Acknowledgments} 

The authors wish to thank the two anonymous reviewers whose valuable 
comments and suggestions helped to improve the overall quality of this paper. 
Furthermore, one of the reviewers recommended an inspection of the
PhD thesis \cite{Keimer2014} which is gratefully acknowledged.
Simone Göttlich was supported by the Deutsche Forschungsgemeinschaft (DFG) within the projects GO1920/10-1 and GO1920/11-1.

\paragraph{Conflict of interest} The authors declare no competing interests.

\bibliographystyle{plainnat}
\bibliography{references}

\appendix

\section{\changed{Special quadratic problems with box constraints}}\label{sec:special_QPs}

\changed{%
Let us fix vectors $\theta,v_\textup{d}\in\R^n$ as well as
$v_\textup{a}\in(\R\cup\{-\infty\})^n$ and $v_\textup{b}\in(\R\cup\{\infty\})^n$
such that all entries of $\theta$ are positive 
while $v_\textup{a}\leq v_\textup{b}$ holds componentwise.
For $\Theta:=\diag(\theta)$, we aim to solve
\begin{equation}\label{eq:trivial_QP}\tag{QP}
	\min\limits_v\{
		\tfrac12 v^\top\Theta v-v_\textup{d}^\top v
		\,|\,
		v\in V_\textup{ad}
	\}
\end{equation}
where $V_\textup{ad}\subset\R^n$ is the box given by
\[
	V_\textup{ad}:=\{v\in\R^n\,|\,v_\textup{a}\leq v\leq v_\textup{b}\}.
\]
First, we observe that the objective function in \eqref{eq:trivial_QP} is uniformly convex
while the feasible set is nonempty, closed, and convex.
Hence, \eqref{eq:trivial_QP} possesses a uniquely determined global minimizer $\bar v\in V_\textup{ad}$.
The latter can be characterized in terms of the necessary and sufficient optimality condition
\begin{equation}\label{eq:finite_dimensional_VI}
	\forall v\in V_\textup{ad}\colon\quad
	(\Theta \bar v-v_\textup{d})^\top(v-\bar v)\geq 0.
\end{equation}
We note that $\Theta$ is a positive definite diagonal matrix.
Hence, it is reasonable to set 
\[
	\tilde v:=\Theta^{-1} v_\textup{d}.
\]
Note that $\tilde v_i=\theta_i^{-1}v_{\textup{d},i}$ holds for all $i=1,\ldots,n$.
We will now show that
\begin{equation}\label{eq:solution_of_QP}
	\bar v=\max(v_\textup{a},\min(\tilde v,v_\textup{b}))
\end{equation}
holds true, i.e., that $\bar v$ is the projection of $\tilde v$ onto the box $V_\textup{ad}$.
Note that $\max$ and $\min$ have to be interpreted componentwise in \eqref{eq:solution_of_QP}.
We introduce index sets $I_\textup{a},I_0,I_\textup{b}\subset\{1,\ldots,n\}$ by means of
\begin{align*}
	I_\textup{a}&:=\{i\in\{1,\ldots,n\}\,|\,\tilde v_i< v_{\textup{a},i}\},
	\\
	I_0&:=\{i\in\{1,\ldots,n\}\,|\,v_{\textup{a},i}\leq\tilde v_i\leq v_{\textup{b},i}\},
	\\
	I_\textup{a}&:=\{i\in\{1,\ldots,n\}\,|\,v_{\textup{b},i}< \tilde v_i\}.
\end{align*}
Clearly, these sets form a disjoint partition of $\{1,\ldots,n\}$, 
and \eqref{eq:solution_of_QP} can be rewritten as
\[
	\forall i\in\{1,\ldots,n\}\colon\quad
	\bar v_i
	=
	\begin{cases}
		v_{\textup{a},i}	&	i\in I_\textup{a},\\
		\tilde v_i			&	i\in I_0,\\
		v_{\textup{b},i}	&	i\in I_\textup{b}.
	\end{cases}
\]
Pick $v\in V_\textup{ad}$ arbitrarily. 
Taking together all of the above findings, we end up with
\begin{align*}
	&(\Theta\bar v-v_\textup{d})^\top(v-\bar v)\\
	&\qquad
	=
	\sum_{i\in I_\textup{a}}(\theta_iv_{\textup{a},i}-v_{\textup{d},i})
		\underbrace{(v_i-v_{\textup{a},i})}_{\geq 0}
	+
	\sum_{i\in I_\textup{b}}(\theta_iv_{\textup{b},i}-v_{\textup{d},i})
		\underbrace{(v_i-v_{\textup{b},i})}_{\leq 0}
	+
	\sum_{i\in I_0}\underbrace{(\theta_i\tilde v_i-v_{\textup{d},i})}_{=0}(v_i-\tilde v_i)
	\\
	&\qquad
	\geq
	\sum_{i\in I_\textup{a}}
		\underbrace{(\theta_i\tilde v_i-v_{\textup{d},i})}_{=0}(v_i-v_{\textup{a},i})
	+
	\sum_{i\in I_\textup{b}}
		\underbrace{(\theta_i\tilde v_i-v_{\textup{d},i})}_{=0}(v_i-v_{\textup{b},i})
	=
	0,
\end{align*}
and this shows that $\bar v$ constructed as in \eqref{eq:solution_of_QP} is, indeed, 
a solution of \eqref{eq:finite_dimensional_VI} and, thus, the
uniquely determined global minimizer of \eqref{eq:trivial_QP}.
}%

\end{document}